\documentclass[10pt]{article}
\usepackage{amsmath}
\usepackage{amssymb}
\usepackage{amsthm}
\usepackage{mathrsfs}
\numberwithin{equation}{section}

\begin{document}

\title{Boundedness of $\theta$-type Calder\'on--Zygmund operators and commutators in the generalized weighted Morrey spaces}
\author{Hua Wang \footnote{E-mail address: wanghua@pku.edu.cn.}\\
\footnotesize{College of Mathematics and Econometrics, Hunan University, Changsha 410082, P. R. China}}
\date{}
\maketitle

\begin{abstract}
In this paper, we first introduce some new Morrey type spaces containing generalized Morrey space and weighted Morrey space as special cases. Then we discuss the strong type and weak type estimates for a class of Calder\'on--Zygmund type operators $T_\theta$ in these new Morrey type spaces. Furthermore, the strong type estimate and endpoint estimate of commutators $[b,T_{\theta}]$ formed by $b$ and $T_{\theta}$ are established. Also we study related problems about two-weight, weak type inequalities for $T_{\theta}$ and $[b,T_{\theta}]$ in the Morrey type spaces and give partial results.
\end{abstract}

\section{Introduction}

Calder\'on--Zygmund singular integral operators and their generalizations on the Euclidean space $\mathbb R^n$ have been extensively studied (see \cite{duoand,garcia,stein2,yabuta} for instance). In particular, Yabuta \cite{yabuta} introduced certain $\theta$-type Calder\'on--Zygmund operators to facilitate his study of certain classes of pseudo-differential operators. Following the terminology of Yabuta \cite{yabuta}, we introduce the so-called $\theta$-type Calder\'on--Zygmund operators.

\newtheorem{defn}{Definition}[section]

\begin{defn}
Let $\theta$ be a non-negative, non-decreasing function on $\mathbb R^+=(0,+\infty)$ with
\begin{equation}\label{theta1}
\int_0^1\frac{\theta(t)}{\,t\,}dt<\infty.
\end{equation}
A measurable function $K(\cdot,\cdot)$ on $\mathbb R^n\times\mathbb R^n\backslash\{(x,x):x\in\mathbb R^n\}$ is said to be a $\theta$-type kernel if it satisfies
\begin{align}
&(i)\quad \big|K(x,y)\big|\leq \frac{C}{|x-y|^{n}},\quad \mbox{for any }\, x\neq y;\\
&(ii)\quad \big|K(x,y)-K(z,y)\big|+\big|K(y,x)-K(y,z)\big|\leq \frac{C}{|x-y|^{n}}\cdot\theta\Big(\frac{|x-z|}{|x-y|}\Big), \\
&\qquad \mbox{for }\, |x-z|<|x-y|/2.\notag
\end{align}
\end{defn}

\begin{defn}
Let $T_\theta$ be a linear operator from $\mathscr S(\mathbb R^n)$ into its dual $\mathscr S'(\mathbb R^n)$. We say that $T_\theta$ is a $\theta$-type Calder\'on--Zygmund operator if

$(1)$ $T_\theta$ can be extended to be a bounded linear operator on $L^2(\mathbb R^n);$

$(2)$ There is a $\theta$-type kernel $K(x,y)$ such that
\begin{equation}
T_\theta f(x):=\int_{\mathbb R^n}K(x,y)f(y)\,dy
\end{equation}
for all $f\in C^\infty_0(\mathbb R^n)$ and for all $x\notin supp\,f$, where $C^\infty_0(\mathbb R^n)$ is the space consisting of all infinitely differentiable functions on $\mathbb R^n$ with compact supports.
\end{defn}
Note that the classical Calder\'on--Zygmund operator with standard kernel (see \cite{duoand,garcia}) is a special case of $\theta$-type operator $T_{\theta}$ when $\theta(t)=t^{\delta}$ with $0<\delta\leq1$.
\begin{defn}
Given a locally integrable function $b$ defined on $\mathbb R^n$, and given a $\theta$-type Calder\'on--Zygmund operator $T_{\theta}$, the linear commutator $[b,T_\theta]$ generated by $b$ and $T_{\theta}$ is defined for smooth, compactly supported functions $f$ as
\begin{equation}
\begin{split}
[b,T_\theta]f(x)&:=b(x)\cdot T_\theta f(x)-T_\theta(b\cdot f)(x)\\
&=\int_{\mathbb R^n}\big[b(x)-b(y)\big]K(x,y)f(y)\,dy.
\end{split}
\end{equation}
\end{defn}

\newtheorem{theorem}{Theorem}[section]

\newtheorem{corollary}{Corollary}[section]

\newtheorem{lemma}{Lemma}[section]

Suppose that $\theta$ is a non-negative, non-decreasing function on $\mathbb R^+=(0,+\infty)$ satisfying the condition (\ref{theta1}) (we always restrict $\theta$ satisfies this condition (\ref{theta1}) in the whole paper), we give the following weighted result of $T_{\theta}$ obtained by Quek and Yang in \cite{quek}.

\begin{theorem}[\cite{quek}]\label{strongweak}
Let $1\leq p<\infty$ and $w\in A_p$. Then the $\theta$-type Calder\'on--Zygmund operator $T_{\theta}$ is bounded on $L^p_w(\mathbb R^n)$ for $p>1$, and bounded from $L^1_w(\mathbb R^n)$ into $WL^1_w(\mathbb R^n)$ for $p=1$.
\end{theorem}

Since linear commutator has a greater degree of singularity than the corresponding $\theta$-type Calder\'on--Zygmund operator, we need a slightly stronger condition (\ref{theta2}) given below. The following weighted endpoint estimate for commutator $[b,T_{\theta}]$ of the $\theta$-type Calder\'on--Zygmund operator was established in \cite{zhang2} under a stronger condition (\ref{theta2}) assumed on $\theta$, if $b\in BMO(\mathbb R^n)$ (for the unweighted case, see \cite{liu}).

\begin{theorem}[\cite{zhang2}]\label{commutator}
Let $\theta$ be a non-negative, non-decreasing function on $\mathbb R^+=(0,+\infty)$ with
\begin{equation}\label{theta2}
\int_0^1\frac{\theta(t)\cdot|\log t|}{t}dt<\infty,
\end{equation}
and let $w\in A_1$ and $b\in BMO(\mathbb R^n)$. Then for all $\sigma>0$, there is a constant $C>0$ independent of $f$ and $\sigma>0$ such that
\begin{equation*}
w\big(\big\{x\in\mathbb R^n:\big|[b,T_\theta](f)(x)\big|>\sigma\big\}\big)
\leq C\int_{\mathbb R^n}\Phi\left(\frac{|f(x)|}{\sigma}\right)\cdot w(x)\,dx,
\end{equation*}
where $\Phi(t)=t\cdot(1+\log^+t)$ and $\log^+t=\max\{\log t,0\}$.
\end{theorem}

On the other hand, the classical Morrey space was originally introduced by Morrey in \cite{morrey} to study the local behavior of solutions to second order elliptic partial differential equations. Since then, this space played an important role in studying the regularity of solutions to partial differential equations. In \cite{mizuhara}, Mizuhara introduced the generalized Morrey space $\mathcal L^{p,\Psi}$ which was later extended and studied by many authors. In \cite{komori}, Komori and Shirai defined a version of the weighted Morrey space $\mathcal L^{p,\kappa}(w)$ which is a natural generalization of the weighted Lebesgue space.Let $T_\theta$ be the $\theta$-type Calder\'on--Zygmund operator, and let $[b,T_{\theta}]$ be its linear commutator. The main purpose of this paper is twofold. We first define a new kind of Morrey type spaces $\mathcal M^{p,\psi}(w)$ containing generalized Morrey space $\mathcal L^{p,\Psi}$ and weighted Morrey space $\mathcal L^{p,\kappa}(w)$ as special cases, and then we will establish the weighted strong type and endpoint estimates for $T_{\theta}$ and $[b,T_{\theta}]$ in these Morrey type spaces $\mathcal M^{p,\psi}(w)$ for all $1\leq p<\infty$ and $w\in A_p$. In addition, we will discuss two-weight, weak type norm inequalities for $T_{\theta}$ and $[b,T_{\theta}]$ in $\mathcal M^{p,\psi}(w)$ and give some partial results.

Throughout this paper $C$ will denote a positive constant whose value may change at each appearance. We also use $A\approx B$ to denote the equivalence of $A$ and $B$; that is, there exist two positive constants $C_1$, $C_2$ independent of $A$, $B$ such that $C_1 A\leq B\leq C_2 A$.

\section{Statements of the main results}

\subsection{Notation and preliminaries}
Let $\mathbb R^n$ be the $n$-dimensional Euclidean space of points $x=(x_1,x_2,\dots,x_n)$ with norm $|x|=(\sum_{i=1}^n x_i^2)^{1/2}$. For $x_0\in\mathbb R^n$ and $r>0$, let $B(x_0,r)=\{x\in\mathbb R^n:|x-x_0|<r\}$ denote the open ball centered at $x_0$ of radius $r$, $B(x_0,r)^c$ denote its complement and $|B(x_0,r)|$ be the Lebesgue measure of the ball $B(x_0,r)$. A weight $w$ is a nonnegative locally integrable function on $\mathbb R^n$ that takes values in $(0,+\infty)$ almost everywhere. A weight $w$ is said to belong to the Muckenhoupt's class $A_p$ for $1<p<\infty$, if there exists a constant $C>0$ such that
\begin{equation*}
\left(\frac1{|B|}\int_B w(x)\,dx\right)^{1/p}\left(\frac1{|B|}\int_B w(x)^{-p'/p}\,dx\right)^{1/{p'}}\leq C
\end{equation*}
for every ball $B\subset\mathbb R^n$, where $p'$ is the dual of $p$ such that $1/p+1/{p'}=1$. The class $A_1$ is defined replacing the above inequality by
\begin{equation*}
\frac1{|B|}\int_B w(x)\,dx\leq C\cdot\underset{x\in B}{\mbox{ess\,inf}}\;w(x)
\end{equation*}
for every ball $B\subset\mathbb R^n$. We also define $A_\infty=\bigcup_{1\leq p<\infty}A_p$. Given a ball $B$ and $\lambda>0$, $\lambda B$ will denote the ball with the same center as $B$ whose radius is $\lambda$ times that of $B$. Given a Lebesgue measurable set $E$ and a weight function $w$, we denote the characteristic function of $E$ by $\chi_E$, the Lebesgue measure of $E$ by $|E|$ and the weighted measure of $E$ by $w(E)$, where $w(E)=\int_E w(x)\,dx$. It is well known that if $w\in A_p$ with $1\leq p<\infty$(or $w\in A_\infty$), then $w$ satisfies the doubling condition; that is, for any ball $B$, there exists an absolute constant $C>0$ such that (see \cite{garcia})
\begin{equation}\label{weights}
w(2B)\leq C\,w(B).
\end{equation}
Moreover, if $w\in A_\infty$, then for any ball $B$ and any measurable subset $E$ of a ball $B$, there exists a number $\delta>0$ independent of $E$ and $B$ such that (see \cite{garcia})
\begin{equation}\label{compare}
\frac{w(E)}{w(B)}\le C\left(\frac{|E|}{|B|}\right)^\delta.
\end{equation}

Given a weight function $w$ on $\mathbb R^n$, as usual, the weighted Lebesgue space $L^p_w(\mathbb R^n)$ for $1\leq p<\infty$ is defined as the set of all functions $f$ such that
\begin{equation*}
\big\|f\big\|_{L^p_w}:=\bigg(\int_{\mathbb R^n}\big|f(x)\big|^pw(x)\,dx\bigg)^{1/p}<\infty.
\end{equation*}
We also denote by $WL^p_w(\mathbb R^n)$($1\leq p<\infty$) the weighted weak Lebesgue space consisting of all measurable functions $f$ such that
\begin{equation*}
\big\|f\big\|_{WL^p_w}:=
\sup_{\lambda>0}\lambda\cdot\Big[w\big(\big\{x\in\mathbb R^n:|f(x)|>\lambda\big\}\big)\Big]^{1/p}<\infty.
\end{equation*}

We next recall some basic definitions and facts about Orlicz spaces needed for the proof of the main results. For further information on the subject, one can see \cite{rao}. A function $\mathcal A$ is called a Young function if it is continuous, nonnegative, convex and strictly increasing on $[0,+\infty)$ with $\mathcal A(0)=0$ and $\mathcal A(t)\to +\infty$ as $t\to +\infty$. An important example of Young function is $\mathcal A(t)=t^p(1+\log^+t)^p$ with some $1\leq p<\infty$. Given a Young function $\mathcal A$, we define the $\mathcal A$-average of a function $f$ over a ball $B$ by means of the following Luxemburg norm:
\begin{equation*}
\big\|f\big\|_{\mathcal A,B}
:=\inf\left\{\lambda>0:\frac{1}{|B|}\int_B\mathcal A\left(\frac{|f(x)|}{\lambda}\right)dx\leq1\right\}.
\end{equation*}
When $\mathcal A(t)=t^p$, $1\leq p<\infty$, it is easy to see that
\begin{equation*}
\big\|f\big\|_{\mathcal A,B}=\left(\frac{1}{|B|}\int_B\big|f(x)\big|^p\,dx\right)^{1/p};
\end{equation*}
that is, the Luxemburg norm coincides with the normalized $L^p$ norm. Given a Young function $\mathcal A$, we use $\bar{\mathcal A}$ to denote the complementary Young function associated to $\mathcal A$. Then the following generalized H\"older's inequality holds for any given ball $B$:
\begin{equation*}
\frac{1}{|B|}\int_B\big|f(x)\cdot g(x)\big|\,dx\leq 2\big\|f\big\|_{\mathcal A,B}\big\|g\big\|_{\bar{\mathcal A},B}.
\end{equation*}
In particular, when $\mathcal A(t)=t\cdot(1+\log^+t)$, we know that its complementary Young function is $\bar{\mathcal A}(t)\approx\exp(t)$. In this situation, we denote
\begin{equation*}
\big\|f\big\|_{L\log L,B}=\big\|f\big\|_{\mathcal A,B}, \qquad
\big\|g\big\|_{\exp L,B}=\big\|g\big\|_{\bar{\mathcal A},B}.
\end{equation*}
So we have
\begin{equation}\label{holder}
\frac{1}{|B|}\int_B\big|f(x)\cdot g(x)\big|\,dx\leq 2\big\|f\big\|_{L\log L,B}\big\|g\big\|_{\exp L,B}.
\end{equation}

Let us now recall the definition of the space of $BMO(\mathbb R^n)$(see \cite{duoand,john}). $BMO(\mathbb R^n)$ is the Banach function space modulo constants with the norm $\|\cdot\|_*$ defined by
\begin{equation*}
\|b\|_*:=\sup_{B}\frac{1}{|B|}\int_B|b(x)-b_B|\,dx<\infty,
\end{equation*}
where the supremum is taken over all balls $B$ in $\mathbb R^n$ and $b_B$ stands for the mean value of $b$ over $B$, that is,
\begin{equation*}
b_B:=\frac{1}{|B|}\int_B b(y)\,dy.
\end{equation*}

\subsection{Morrey type spaces}
Let us begin with the definitions of the weighted Morrey space and generalized Morrey space.
\begin{defn}[\cite{komori}]
Let $1\leq p<\infty$, $0<\kappa<1$ and $w$ be a weight function on $\mathbb R^n$. Then the weighted Morrey space $\mathcal L^{p,\kappa}(w)$ is defined by
\begin{equation*}
\mathcal L^{p,\kappa}(w):=\left\{f\in L^p_{loc}(w):\big\|f\big\|_{\mathcal L^{p,\kappa}(w)}
=\sup_B\left(\frac{1}{w(B)^{\kappa}}\int_B|f(x)|^pw(x)\,dx\right)^{1/p}<\infty\right\},
\end{equation*}
where the supremum is taken over all balls $B$ in $\mathbb R^n$. We also denote by $W\mathcal L^{1,\kappa}(w)$ the weighted weak Morrey space of all measurable functions $f$ such that
\begin{equation*}
\sup_B\sup_{\lambda>0}\frac{1}{w(B)^{\kappa}}\lambda\cdot w\big(\big\{x\in B:|f(x)|>\lambda\big\}\big)
\leq C<\infty.
\end{equation*}
\end{defn}

Let $\Psi=\Psi(r)$, $r>0$, be a growth function, that is, a positive increasing function in $(0,+\infty)$ and satisfy the following doubling condition:
\begin{equation}\label{doubling}
\Psi(2r)\leq D\cdot\Psi(r), \qquad \mbox{for all }\,r>0,
\end{equation}
where $D=D(\Psi)\ge1$ is a doubling constant independent of $r$.

\begin{defn}[\cite{mizuhara}]
Let $1\leq p<\infty$ and $\Psi$ be a growth function in $(0,+\infty)$. Then the generalized Morrey space $\mathcal L^{p,\Psi}(\mathbb R^n)$ is defined by
\begin{equation*}
\mathcal L^{p,\Psi}(\mathbb R^n):=\left\{f\in L^p_{loc}(\mathbb R^n):\big\|f\big\|_{\mathcal L^{p,\Psi}}
=\sup_{r>0;B(x_0,r)}\bigg(\frac{1}{\Psi(r)}\int_{B(x_0,r)}|f(x)|^p\,dx\bigg)^{1/p}<\infty\right\},
\end{equation*}
where the supremum is taken over all balls $B(x_0,r)$ in $\mathbb R^n$ with $x_0\in\mathbb R^n$. We also denote by $W\mathcal L^{1,\Psi}(\mathbb R^n)$ the generalized weak Morrey space of all measurable functions $f$ for which
\begin{equation*}
\sup_{B(x_0,r)}\sup_{\lambda>0}\frac{1}{\Psi(r)}\lambda\cdot\big|\big\{x\in B(x_0,r):|f(x)|>\lambda\big\}\big|
\leq C<\infty.
\end{equation*}
\end{defn}

In order to unify these two definitions, we now introduce Morrey type spaces associated to $\psi$ as follows.
Let $0\leq\kappa<1$. Assume that $\psi(\cdot)$ is a positive increasing function defined in $(0,+\infty)$ and satisfies the following $\mathcal D_\kappa$ condition:
\begin{equation}\label{D condition}
\frac{\psi(\xi)}{\xi^\kappa}\leq C\cdot\frac{\psi(\xi')}{(\xi')^\kappa},
\qquad \mbox{for any}\;0<\xi'<\xi<+\infty,
\end{equation}
where $C>0$ is a constant independent of $\xi$ and $\xi'$.

\begin{defn}\label{WMorrey}
Let $1\leq p<\infty$, $0\leq\kappa<1$, $\psi$ satisfy the $\mathcal D_\kappa$ condition $(\ref{D condition})$ and $w$ be a weight function on $\mathbb R^n$. We denote by $\mathcal M^{p,\psi}(w)$ the generalized weighted Morrey space, the space of all locally integrable functions $f$ with finite norm
\begin{equation*}
\big\|f\big\|_{\mathcal M^{p,\psi}(w)}
:=\sup_B\left(\frac{1}{\psi(w(B))}\int_B |f(x)|^pw(x)\,dx\right)^{1/p}\leq C<\infty.
\end{equation*}
Then we know that $\mathcal M^{p,\psi}(w)$ becomes a Banach function space with respect to the norm $\|\cdot\|_{\mathcal M^{p,\psi}(w)}$. Furthermore, we denote by $W\mathcal M^{p,\psi}(w)$ the generalized weighted weak Morrey space of all measurable functions $f$ for which
\begin{equation*}
\big\|f\big\|_{W\mathcal M^{p,\psi}(w)}:=\sup_B\sup_{\sigma>0}\frac{1}{\psi(w(B))^{1/p}}\sigma
\cdot \Big[w\big(\big\{x\in B:|f(x)|>\sigma\big\}\big)\Big]^{1/p}\leq C<\infty.
\end{equation*}
\end{defn}

\begin{defn}
In the unweighted case (when $w$ equals a constant function), we denote the generalized unweighted Morrey space by $\mathcal M^{p,\psi}(\mathbb R^n)$ and weak Morrey space by $W\mathcal M^{p,\psi}(\mathbb R^n)$. That is, let $1\leq p<\infty$ and $\psi$ satisfy the $\mathcal D_\kappa$ condition $(\ref{D condition})$ with $0\leq\kappa<1$, we define
\begin{equation*}
\mathcal M^{p,\psi}(\mathbb R^n):=\left\{f\in L^p_{loc}(\mathbb R^n):
\big\|f\big\|_{\mathcal M^{p,\psi}}=\sup_B\bigg(\frac{1}{\psi(|B|)}\int_B |f(x)|^p\,dx\bigg)^{1/p}<\infty\right\},
\end{equation*}
and
\begin{equation*}
W\mathcal M^{p,\psi}(\mathbb R^n):=\left\{f:
\big\|f\big\|_{W\mathcal M^{p,\psi}}=\sup_B\sup_{\sigma>0}\frac{1}{\psi(|B|)^{1/p}}\sigma
\cdot \Big|\big\{x\in B:|f(x)|>\sigma\big\}\Big|^{1/p}<\infty\right\}.
\end{equation*}
\end{defn}

Note that
\begin{itemize}
  \item If $\psi(x)\equiv 1$, then $\mathcal M^{p,\psi}(w)=L^p_w(\mathbb R^n)$ and $W\mathcal M^{p,\psi}(w)=WL^p_w(\mathbb R^n)$. Thus our (weak) Morrey type space is an extension of the weighted (weak) Lebesgue space;
  \item If $\psi(x)=x^{\kappa}$ with $0<\kappa<1$, then $\mathcal M^{p,\psi}(w)$ is just the weighted Morrey space $\mathcal L^{p,\kappa}(w)$, and $W\mathcal M^{1,\psi}(w)$ is just the weighted weak Morrey space $W\mathcal L^{1,\kappa}(w)$;
  \item If $w(x)\equiv 1$, below we will show that $\mathcal M^{p,\psi}(\mathbb R^n)$ reduces to the generalized Morrey space $\mathcal L^{p,\Psi}(\mathbb R^n)$, and $W\mathcal M^{1,\psi}(\mathbb R^n)$ reduces to the generalized weak Morrey space $W\mathcal L^{1,\Psi}(\mathbb R^n)$.
\end{itemize}

Our main results on the boundedness of $T_{\theta}$ in the Morrey type spaces $\mathcal M^{p,\psi}(w)$ can be formulated as follows.

\begin{theorem}\label{mainthm:1}
Let $1<p<\infty$ and $w\in A_p$. Assume that $\psi$ satisfies the $\mathcal D_\kappa$ condition $(\ref{D condition})$ with $0\leq\kappa<1$, then the $\theta$-type Calder\'on--Zygmund operator $T_{\theta}$ is bounded on $\mathcal M^{p,\psi}(w)$.
\end{theorem}

\begin{theorem}\label{mainthm:2}
Let $p=1$ and $w\in A_1$. Assume that $\psi$ satisfies the $\mathcal D_\kappa$ condition $(\ref{D condition})$ with $0\leq\kappa<1$, then the $\theta$-type Calder\'on--Zygmund operator $T_{\theta}$ is bounded from $\mathcal M^{1,\psi}(w)$ into $W\mathcal M^{1,\psi}(w)$.
\end{theorem}

Let $\theta$ be a non-negative, non-decreasing function on $\mathbb R^+=(0,+\infty)$ satisfying the condition $(\ref{theta2})$, and let $[b,T_{\theta}]$ be the commutator formed by $T_{\theta}$ and BMO function $b$. For the strong type estimate of the linear commutator $[b,T_{\theta}]$ in $\mathcal M^{p,\psi}(w)$ with $1<p<\infty$, we will prove

\begin{theorem}\label{mainthm:3}
Let $1<p<\infty$, $w\in A_p$ and $b\in BMO(\mathbb R^n)$. Assume that $\theta$ satisfies $(\ref{theta2})$ and $\psi$ satisfies the $\mathcal D_\kappa$ condition $(\ref{D condition})$ with $0\leq\kappa<1$, then the commutator operator $[b,T_{\theta}]$ is bounded on $\mathcal M^{p,\psi}(w)$.
\end{theorem}

To obtain endpoint estimate for the linear commutator $[b,T_{\theta}]$ in $\mathcal M^{1,\psi}(w)$, we first need to define the weighted $\mathcal A$-average of a function $f$ over a ball $B$ by means of the weighted Luxemburg norm; that is, given a Young function $\mathcal A$ and $w\in A_\infty$, we define (see \cite{rao,zhang})
\begin{equation*}
\big\|f\big\|_{\mathcal A(w),B}:=\inf\left\{\sigma>0:\frac{1}{w(B)}
\int_B\mathcal A\left(\frac{|f(x)|}{\sigma}\right)\cdot w(x)\,dx\leq1\right\}.
\end{equation*}
When $\mathcal A(t)=t$, this norm is denoted by $\|\cdot\|_{L(w),B}$, when $\mathcal A(t)=t\cdot(1+\log^+t)$, this norm is also denoted by $\|\cdot\|_{L\log L(w),B}$. The complementary Young function of $t\cdot(1+\log^+t)$ is $\exp t$ with mean Luxemburg norm denoted by $\|\cdot\|_{\exp L(w),B}$. For $w\in A_\infty$ and for every ball $B$ in $\mathbb R^n$, we can also show the weighted version of \eqref{holder}. Namely, the following generalized H\"older's inequality in the weighted setting
\begin{equation}\label{Wholder}
\frac{1}{w(B)}\int_B|f(x)\cdot g(x)|w(x)\,dx\leq C\big\|f\big\|_{L\log L(w),B}\big\|g\big\|_{\exp L(w),B}
\end{equation}
is valid (see \cite{zhang} for instance). Now we introduce new Morrey type spaces of $L\log L$ type associated to $\psi$ as follows.

\begin{defn}
Let $p=1$, $0\leq\kappa<1$, $\psi$ satisfy the $\mathcal D_\kappa$ condition $(\ref{D condition})$ and $w$ be a weight function on $\mathbb R^n$. We denote by $\mathcal M^{1,\psi}_{L\log L}(w)$ the generalized weighted Morrey space of $L\log L$ type, the space of all locally integrable functions $f$ defined on $\mathbb R^n$ with finite norm $\big\|f\big\|_{\mathcal M^{1,\psi}_{L\log L}(w)}$.
\begin{equation*}
\mathcal M^{1,\psi}_{L\log L}(w):=\left\{f\in L^1_{loc}(w):\big\|f\big\|_{\mathcal M^{1,\psi}_{L\log L}(w)}<\infty\right\},
\end{equation*}
where
\begin{equation*}
\big\|f\big\|_{\mathcal M^{1,\psi}_{L\log L}(w)}
:=\sup_B\left\{\frac{w(B)}{\psi(w(B))}\cdot\big\|f\big\|_{L\log L(w),B}\right\}.
\end{equation*}
\end{defn}

Note that $t\leq t\cdot(1+\log^+t)$ for all $t>0$, then for any ball $B\subset\mathbb R^n$ and $w\in A_\infty$, we have $\big\|f\big\|_{L(w),B}\leq \big\|f\big\|_{L\log L(w),B}$ by definition, i.e., the inequality
\begin{equation}\label{main esti1}
\big\|f\big\|_{L(w),B}=\frac{1}{w(B)}\int_B|f(x)|\cdot w(x)\,dx\leq\big\|f\big\|_{L\log L(w),B}
\end{equation}
holds for any ball $B\subset\mathbb R^n$. From this, we can further see that when $\psi$ satisfies the $\mathcal D_\kappa$ condition $(\ref{D condition})$ with $0\leq\kappa<1$,
\begin{equation}\label{main esti2}
\begin{split}
\frac{1}{\psi(w(B))}\int_B|f(x)|\cdot w(x)\,dx
&=\frac{w(B)}{\psi(w(B))}\cdot\frac{1}{w(B)}\int_B|f(x)|\cdot w(x)\,dx\\
&=\frac{w(B)}{\psi(w(B))}\cdot\big\|f\big\|_{L(w),B}\\
&\leq\frac{w(B)}{\psi(w(B))}\cdot\big\|f\big\|_{L\log L(w),B}.
\end{split}
\end{equation}
Hence, we have $\mathcal M^{1,\psi}_{L\log L}(w)\subset\mathcal M^{1,\psi}(w)$ by definition.
\begin{defn}
In the unweighted case (when $w$ equals a constant function), we denote by $\mathcal M^{1,\psi}_{L\log L}(\mathbb R^n)$ the generalized unweighted Morrey space of $L\log L$ type. That is, let $p=1$ and $\psi$ satisfy the $\mathcal D_\kappa$ condition $(\ref{D condition})$ with $0\leq\kappa<1$, we define
\begin{equation*}
\mathcal M^{1,\psi}_{L\log L}(\mathbb R^n):=\left\{f\in L^1_{loc}(\mathbb R^n):
\big\|f\big\|_{\mathcal M^{1,\psi}_{L\log L}(\mathbb R^n)}<\infty\right\},
\end{equation*}
where
\begin{equation*}
\big\|f\big\|_{\mathcal M^{1,\psi}_{L\log L}(\mathbb R^n)}
:=\sup_B\left\{\frac{|B|}{\psi(|B|)}\cdot\big\|f\big\|_{L\log L,B}\right\}.
\end{equation*}
\end{defn}

We also consider the special case when $\psi$ is taken to be $\psi(x)=x^\kappa$ with $0<\kappa<1$, and denote the corresponding space by $\mathcal L^{1,\kappa}_{L\log L}(w)$.
\begin{defn}
Let $p=1$, $0<\kappa<1$ and $w$ be a weight function on $\mathbb R^n$. We denote by $\mathcal L^{1,\kappa}_{L\log L}(w)$ the weighted Morrey space of $L\log L$ type, the space of all locally integrable functions $f$ defined on $\mathbb R^n$ with finite norm $\big\|f\big\|_{\mathcal L^{1,\kappa}_{L\log L}(w)}$.
\begin{equation*}
\mathcal L^{1,\kappa}_{L\log L}(w):=\left\{f\in L^1_{loc}(w):\big\|f\big\|_{\mathcal L^{1,\kappa}_{L\log L}(w)}<\infty\right\},
\end{equation*}
where
\begin{equation*}
\big\|f\big\|_{\mathcal L^{1,\kappa}_{L\log L}(w)}
:=\sup_B\left\{w(B)^{1-\kappa}\cdot\big\|f\big\|_{L\log L(w),B}\right\}.
\end{equation*}
In this situation, we have $\mathcal L^{1,\kappa}_{L\log L}(w)\subset\mathcal L^{1,\kappa}(w)$.
\end{defn}

For the endpoint case, we will also prove the following weak type $L\log L$ estimate of the linear commutator $[b,T_{\theta}]$ in the Morrey type space associated to $\psi$.

\begin{theorem}\label{mainthm:4}
Let $p=1$, $w\in A_1$ and $b\in BMO(\mathbb R^n)$. Assume that $\theta$ satisfies $(\ref{theta2})$ and $\psi$ satisfies the $\mathcal D_\kappa$ condition $(\ref{D condition})$ with $0\leq\kappa<1$, then for any given $\sigma>0$ and any ball $B\subset\mathbb R^n$, there exists a constant $C>0$ independent of $f$, $B$ and $\sigma>0$ such that
\begin{equation*}
\frac{1}{\psi(w(B))}\cdot w\big(\big\{x\in B:\big|[b,T_{\theta}](f)(x)\big|>\sigma\big\}\big)
\leq C\cdot\bigg\|\Phi\left(\frac{|f|}{\,\sigma\,}\right)\bigg\|_{\mathcal M^{1,\psi}_{L\log L}(w)},
\end{equation*}
where $\Phi(t)=t\cdot(1+\log^+t)$. From the definitions, we can roughly say that the commutator operator $[b,T_{\theta}]$ is bounded from $\mathcal M^{1,\psi}_{L\log L}(w)$ into $W\mathcal M^{1,\psi}(w)$.
\end{theorem}

In particular, if we take $\psi(x)=x^\kappa$ with $0<\kappa<1$, then we immediately
get the following strong type estimate and endpoint estimate of $T_{\theta}$ and $[b,T_{\theta}]$ in the weighted Morrey spaces $\mathcal L^{p,\kappa}(w)$ for all $0<\kappa<1$ and $1\leq p<\infty$.

\begin{corollary}\label{cor:1}
Let $1<p<\infty$, $0<\kappa<1$ and $w\in A_p$. Then the $\theta$-type Calder\'on--Zygmund operator $T_{\theta}$ is bounded on $\mathcal L^{p,\kappa}(w)$.
\end{corollary}

\begin{corollary}\label{cor:2}
Let $p=1$, $0<\kappa<1$ and $w\in A_1$. Then the $\theta$-type Calder\'on--Zygmund operator $T_{\theta}$ is bounded from $\mathcal L^{1,\kappa}(w)$ into $W\mathcal L^{1,\kappa}(w)$.
\end{corollary}

\begin{corollary}\label{cor:3}
Let $1<p<\infty$, $0<\kappa<1$, $w\in A_p$ and $b\in BMO(\mathbb R^n)$. Assume that $\theta$ satisfies $(\ref{theta2})$, then the commutator operator $[b,T_{\theta}]$ is bounded on $\mathcal L^{p,\kappa}(w)$.
\end{corollary}

\begin{corollary}\label{cor:4}
Let $p=1$, $0<\kappa<1$, $w\in A_1$ and $b\in BMO(\mathbb R^n)$. Assume that $\theta$ satisfies $(\ref{theta2})$, then for any given $\sigma>0$ and any ball $B\subset\mathbb R^n$, there exists a constant $C>0$ independent of $f$, $B$ and $\sigma>0$ such that
\begin{equation*}
\frac{1}{w(B)^{\kappa}}\cdot w\big(\big\{x\in B:\big|[b,T_{\theta}](f)(x)\big|>\sigma\big\}\big)
\leq C\cdot\bigg\|\Phi\left(\frac{|f|}{\,\sigma\,}\right)\bigg\|_{\mathcal L^{1,\kappa}_{L\log L}(w)},
\end{equation*}
where $\Phi(t)=t\cdot(1+\log^+t)$.
\end{corollary}

Naturally, when $w(x)\equiv 1$ we have the following unweighted results.
\begin{corollary}\label{cor:5}
Let $1<p<\infty$. Assume that $\psi$ satisfies the $\mathcal D_\kappa$ condition $(\ref{D condition})$ with $0\leq\kappa<1$, then the $\theta$-type Calder\'on--Zygmund operator $T_{\theta}$ is bounded on $\mathcal M^{p,\psi}(\mathbb R^n)$.
\end{corollary}

\begin{corollary}\label{cor:6}
Let $p=1$. Assume that $\psi$ satisfies the $\mathcal D_\kappa$ condition $(\ref{D condition})$ with $0\leq\kappa<1$, then the $\theta$-type Calder\'on--Zygmund operator $T_{\theta}$ is bounded from $\mathcal M^{1,\psi}(\mathbb R^n)$ into $W\mathcal M^{1,\psi}(\mathbb R^n)$.
\end{corollary}

\begin{corollary}\label{cor:7}
Let $1<p<\infty$ and $b\in BMO(\mathbb R^n)$. Assume that $\theta$ satisfies $(\ref{theta2})$ and $\psi$ satisfies the $\mathcal D_\kappa$ condition $(\ref{D condition})$ with $0\leq\kappa<1$, then the commutator operator $[b,T_{\theta}]$ is bounded on $\mathcal M^{p,\psi}(\mathbb R^n)$.
\end{corollary}

\begin{corollary}\label{cor:8}
Let $p=1$ and $b\in BMO(\mathbb R^n)$. Assume that $\theta$ satisfies $(\ref{theta2})$ and $\psi$ satisfies the $\mathcal D_\kappa$ condition $(\ref{D condition})$ with $0\leq\kappa<1$, then for any given $\sigma>0$ and any ball $B\subset\mathbb R^n$, there exists a constant $C>0$ independent of $f$, $B$ and $\sigma>0$ such that
\begin{equation*}
\frac{1}{\psi(|B|)}\cdot\Big|\big\{x\in B:\big|[b,T_{\theta}](f)(x)\big|>\sigma\big\}\Big|
\leq C\cdot\bigg\|\Phi\left(\frac{|f|}{\,\sigma\,}\right)\bigg\|_{\mathcal M^{1,\psi}_{L\log L}(\mathbb R^n)},
\end{equation*}
where $\Phi(t)=t\cdot(1+\log^+t)$.
\end{corollary}

Let $\Psi=\Psi(r)$, $r>0$, be a growth function with doubling constant $D(\Psi):1\le D(\Psi)<2^n$. If for any fixed $x_0\in\mathbb R^n$ and $r>0$, we set $\psi(|B(x_0,r)|)=\Psi(r)$. Then
\begin{equation*}
\psi\big(2^n|B(x_0,r)|\big)=\psi\big(|B(x_0,2r)|\big)=\Psi(2r).
\end{equation*}
For the doubling constant $D(\Psi)$ satisfying $1\le D(\Psi)<2^n$, which means that $D(\Psi)=2^{\kappa\cdot n}$ for some $0\leq\kappa<1$, then we are able to verify that $\psi$ is an increasing function and satisfies the $\mathcal D_\kappa$ condition $(\ref{D condition})$ with some $0\leq\kappa<1$.

\begin{defn}
Let $p=1$ and $\Psi$ be a growth function in $(0,+\infty)$. We denote by $\mathcal L^{1,\Psi}_{L\log L}(\mathbb R^n)$ the generalized Morrey space of $L\log L$ type, which is defined by
\begin{equation*}
\mathcal L^{1,\Psi}_{L\log L}(\mathbb R^n):=\left\{f\in L^1_{loc}(\mathbb R^n):
\big\|f\big\|_{\mathcal L^{1,\Psi}_{L\log L}(\mathbb R^n)}<\infty\right\},
\end{equation*}
where
\begin{equation*}
\big\|f\big\|_{\mathcal L^{1,\Psi}_{L\log L}(\mathbb R^n)}
:=\sup_{r>0;B(x_0,r)}\left\{\frac{|B(x_0,r)|}{\Psi(r)}\cdot\big\|f\big\|_{L\log L,B(x_0,r)}\right\}.
\end{equation*}
In this situation, we also have $\mathcal L^{1,\Psi}_{L\log L}(\mathbb R^n)\subset\mathcal L^{1,\Psi}(\mathbb R^n)$.
\end{defn}
From the definitions given above, we get $\mathcal M^{p,\psi}(\mathbb R^n)=\mathcal L^{p,\Psi}(\mathbb R^n)$, $W\mathcal M^{1,\psi}(\mathbb R^n)=W\mathcal L^{1,\Psi}(\mathbb R^n)$ and $\mathcal M^{1,\psi}_{L\log L}(\mathbb R^n)=\mathcal L^{1,\Psi}_{L\log L}(\mathbb R^n)$ by the choice of $\Psi$.
Thus, by the above unweighted results (Corollaries 2.5--2.8), we can also obtain strong type estimate and endpoint estimate of $T_{\theta}$ and $[b,T_{\theta}]$ in the generalized Morrey spaces $\mathcal L^{p,\Psi}(\mathbb R^n)$ when $1\leq p<\infty$ and $\Psi$ satisfies the doubling condition $(\ref{doubling})$.

\begin{corollary}\label{cor:5}
Let $1<p<\infty$. Suppose that $\Psi$ satisfies the doubling condition $(\ref{doubling})$ and $1\le D(\Psi)<2^n$, then the $\theta$-type Calder\'on--Zygmund operator $T_{\theta}$ is bounded on $\mathcal L^{p,\Psi}(\mathbb R^n)$.
\end{corollary}

\begin{corollary}\label{cor:6}
Let $p=1$.Suppose that $\Psi$ satisfies the doubling condition $(\ref{doubling})$ and $1\le D(\Psi)<2^n$, then the $\theta$-type Calder\'on--Zygmund operator $T_{\theta}$ is bounded from $\mathcal L^{1,\Psi}(\mathbb R^n)$ into $W\mathcal L^{1,\Psi}(\mathbb R^n)$.
\end{corollary}

\begin{corollary}\label{cor:7}
Let $1<p<\infty$ and $b\in BMO(\mathbb R^n)$. Suppose that $\theta$ satisfies $(\ref{theta2})$ and $\Psi$ satisfies the doubling condition $(\ref{doubling})$ with $1\le D(\Psi)<2^n$, then the commutator operator $[b,T_{\theta}]$ is bounded on $\mathcal L^{p,\Psi}(\mathbb R^n)$.
\end{corollary}

\begin{corollary}\label{cor:8}
Let $p=1$ and $b\in BMO(\mathbb R^n)$. Suppose that $\theta$ satisfies $(\ref{theta2})$ and $\Psi$ satisfies the doubling condition $(\ref{doubling})$ with $1\le D(\Psi)<2^n$, then for any given $\sigma>0$ and any ball $B(x_0,r)\subset\mathbb R^n$, there exists a constant $C>0$ independent of $f$, $B(x_0,r)$ and $\sigma>0$ such that
\begin{equation*}
\frac{1}{\Psi(r)}\cdot\Big|\big\{x\in B(x_0,r):\big|[b,T_{\theta}](f)(x)\big|>\sigma\big\}\Big|
\leq C\cdot\bigg\|\Phi\left(\frac{|f|}{\,\sigma\,}\right)\bigg\|_{\mathcal L^{1,\Psi}_{L\log L}(\mathbb R^n)},
\end{equation*}
where $\Phi(t)=t\cdot(1+\log^+t)$.
\end{corollary}

\section{Proof of Theorems \ref{mainthm:1} and \ref{mainthm:2}}

\begin{proof}[Proof of Theorem $\ref{mainthm:1}$]
Let $f\in\mathcal M^{p,\psi}(w)$ with $1<p<\infty$ and $w\in A_p$. For an arbitrary point $x_0\in\mathbb R^n$, set $B=B(x_0,r_B)$ for the ball centered at $x_0$ and of radius $r_B$, $2B=B(x_0,2r_B)$. We represent $f$ as
\begin{equation*}
f=f\cdot\chi_{2B}+f\cdot\chi_{(2B)^c}:=f_1+f_2;
\end{equation*}
by the linearity of the $\theta$-type Calder\'on--Zygmund operator $T_{\theta}$, we write
\begin{equation*}
\begin{split}
&\frac{1}{\psi(w(B))^{1/p}}\bigg(\int_B\big|T_\theta(f)(x)\big|^pw(x)\,dx\bigg)^{1/p}\\
\leq\ &\frac{1}{\psi(w(B))^{1/p}}\bigg(\int_B\big|T_\theta(f_1)(x)\big|^pw(x)\,dx\bigg)^{1/p}\\
&+\frac{1}{\psi(w(B))^{1/p}}\bigg(\int_B\big|T_\theta(f_2)(x)\big|^pw(x)\,dx\bigg)^{1/p}\\
:=\ &I_1+I_2.
\end{split}
\end{equation*}
Below we will give the estimates of $I_1$ and $I_2$, respectively. By the weighted $L^p$ boundedness of $T_{\theta}$ (see Theorem \ref{strongweak}), we have
\begin{equation*}
\begin{split}
I_1&\leq \frac{1}{\psi(w(B))^{1/p}}
\big\|T_\theta(f_1)\big\|_{L^p_w}\\
&\leq C\cdot\frac{1}{\psi(w(B))^{1/p}}
\bigg(\int_{2B}\big|f(x)\big|^p w(x)\,dx\bigg)^{1/p}\\
&\leq C\big\|f\big\|_{\mathcal M^{p,\psi}(w)}
\cdot\frac{\psi(w(2B))^{1/p}}{\psi(w(B))^{1/p}}.
\end{split}
\end{equation*}
Moreover, since $0<w(B)<w(2B)<+\infty$ when $w\in A_p$ with $1<p<\infty$, then by the $\mathcal D_\kappa$ condition (\ref{D condition}) of $\psi$ and inequality (\ref{weights}), we obtain
\begin{equation*}
\begin{split}
I_1&\leq C\big\|f\big\|_{\mathcal M^{p,\psi}(w)}\cdot\frac{w(2B)^{\kappa/p}}{w(B)^{\kappa/p}}\\
&\leq C\big\|f\big\|_{\mathcal M^{p,\psi}(w)}.
\end{split}
\end{equation*}
As for the term $I_2$, it is clear that when $x\in B$ and $y\in(2B)^c$, we get $|x-y|\approx|x_0-y|$. We then decompose $\mathbb R^n$ into a geometrically increasing sequence of concentric balls, and obtain the following pointwise estimate
\begin{equation}\label{pointwise1}
\begin{split}
\big|T_\theta(f_2)(x)\big|&\leq\int_{\mathbb R^n}\frac{|f_2(y)|}{|x-y|^n}dy
\leq C\int_{(2B)^c}\frac{|f(y)|}{|x_0-y|^n}dy\\
&\leq C\sum_{j=1}^\infty\frac{1}{|2^{j+1}B|}\int_{2^{j+1}B}\big|f(y)\big|\,dy.
\end{split}
\end{equation}
From this, it follows that
\begin{equation*}
I_2\leq C\cdot\frac{w(B)^{1/p}}{\psi(w(B))^{1/p}}\sum_{j=1}^\infty
\frac{1}{|2^{j+1}B|}\int_{2^{j+1}B}\big|f(y)\big|\,dy.
\end{equation*}
By using H\"older's inequality and $A_p$ condition on $w$, we get
\begin{equation*}
\begin{split}
\frac{1}{|2^{j+1}B|}\int_{2^{j+1}B}\big|f(y)\big|\,dy
&\leq\frac{1}{|2^{j+1}B|}\bigg(\int_{2^{j+1}B}\big|f(y)\big|^pw(y)\,dy\bigg)^{1/p}
\left(\int_{2^{j+1}B}w(y)^{-{p'}/p}\,dy\right)^{1/{p'}}\\
&\leq C\big\|f\big\|_{\mathcal M^{p,\psi}(w)}\cdot\frac{\psi(w(2^{j+1}B))^{1/p}}{w(2^{j+1}B)^{1/p}}.
\end{split}
\end{equation*}
Hence
\begin{equation*}
\begin{split}
I_2&\leq C\big\|f\big\|_{\mathcal M^{p,\psi}(w)}
\times\sum_{j=1}^\infty\frac{\psi(w(2^{j+1}B))^{1/p}}{\psi(w(B))^{1/p}}\cdot\frac{w(B)^{1/p}}{w(2^{j+1}B)^{1/p}}.
\end{split}
\end{equation*}
Notice that $w\in A_p\subset A_\infty$ for $1<p<\infty$, then by using the $\mathcal D_\kappa$ condition (\ref{D condition}) of $\psi$ again, the inequality (\ref{compare}) and the fact that $0\leq\kappa<1$, we find that
\begin{align}\label{psi1}
\sum_{j=1}^\infty\frac{\psi(w(2^{j+1}B))^{1/p}}{\psi(w(B))^{1/p}}\cdot\frac{w(B)^{1/p}}{w(2^{j+1}B)^{1/p}}
&\leq C\sum_{j=1}^\infty\frac{w(B)^{{(1-\kappa)}/p}}{w(2^{j+1}B)^{{(1-\kappa)}/p}}\notag\\
&\leq C\sum_{j=1}^\infty\left(\frac{|B|}{|2^{j+1}B|}\right)^{\delta {(1-\kappa)}/p}\notag\\
&\leq C\sum_{j=1}^\infty\left(\frac{1}{2^{(j+1)n}}\right)^{\delta {(1-\kappa)}/p}\notag\\
&\leq C,
\end{align}
which gives our desired estimate $I_2\leq C\big\|f\big\|_{\mathcal M^{p,\psi}(w)}$. Combining the above estimates for $I_1$ and $I_2$, and then taking the supremum over all balls $B\subset\mathbb R^n$, we complete the proof of Theorem \ref{mainthm:1}.
\end{proof}

\begin{proof}[Proof of Theorem $\ref{mainthm:2}$]
Let $f\in\mathcal M^{1,\psi}(w)$ with $w\in A_1$. For an arbitrary ball $B=B(x_0,r_B)\subset\mathbb R^n$, we represent $f$ as
\begin{equation*}
f=f\cdot\chi_{2B}+f\cdot\chi_{(2B)^c}:=f_1+f_2;
\end{equation*}
then for any given $\sigma>0$, by the linearity of the $\theta$-type Calder\'on--Zygmund operator $T_{\theta}$, one can write
\begin{equation*}
\begin{split}
&\frac{1}{\psi(w(B))}\sigma\cdot w\big(\big\{x\in B:\big|T_{\theta}(f)(x)\big|>\sigma\big\}\big)\\
\leq\ &\frac{1}{\psi(w(B))}\sigma\cdot w\big(\big\{x\in B:\big|T_{\theta}(f_1)(x)\big|>\sigma/2\big\}\big)\\
&+\frac{1}{\psi(w(B))}\sigma\cdot w\big(\big\{x\in B:\big|T_{\theta}(f_2)(x)\big|>\sigma/2\big\}\big)\\
:=\ &I'_1+I'_2.
\end{split}
\end{equation*}
We first consider the term $I'_1$. By the weighted weak $(1,1)$ boundedness of $T_{\theta}$ (see Theorem \ref{strongweak}), we have
\begin{equation*}
\begin{split}
I'_1&\leq C\cdot\frac{1}{\psi(w(B))}\big\|f_1\big\|_{L^1_w}\\
&=C\cdot\frac{1}{\psi(w(B))}\bigg(\int_{2B}\big|f(x)\big|w(x)\,dx\bigg)\\
&\leq C\big\|f\big\|_{\mathcal M^{1,\psi}(w)}
\cdot\frac{\psi(w(2B))}{\psi(w(B))}.
\end{split}
\end{equation*}
Moreover, since $0<w(B)<w(2B)<+\infty$ when $w\in A_1$, then we apply the $\mathcal D_\kappa$ condition (\ref{D condition}) of $\psi$ and inequality (\ref{weights}) to obtain that
\begin{equation*}
\begin{split}
I'_1&\leq C\big\|f\big\|_{\mathcal M^{1,\psi}(w)}
\cdot\frac{w(2B)^{\kappa}}{w(B)^{\kappa}}\\
&\leq C\big\|f\big\|_{\mathcal M^{1,\psi}(w)}.
\end{split}
\end{equation*}
As for the term $I'_2$, it follows directly from Chebyshev's inequality and the pointwise estimate \eqref{pointwise1} that
\begin{equation*}
\begin{split}
I'_2&\leq\frac{1}{\psi(w(B))}\sigma\cdot\frac{\,2\,}{\sigma}\int_B\big|T_\theta(f_2)(x)\big|w(x)\,dx\\
&\leq C\cdot\frac{w(B)}{\psi(w(B))}\sum_{j=1}^\infty\frac{1}{|2^{j+1}B|}\int_{2^{j+1}B}\big|f(y)\big|\,dy.
\end{split}
\end{equation*}
Another application of $A_1$ condition on $w$ gives that
\begin{equation*}
\begin{split}
\frac{1}{|2^{j+1}B|}\int_{2^{j+1}B}\big|f(y)\big|\,dy
&\leq C\frac{1}{w(2^{j+1}B)}\cdot\underset{y\in 2^{j+1}B}{\mbox{ess\,inf}}\;w(y)\int_{2^{j+1}B}\big|f(y)\big|\,dy\\
&\leq C\frac{1}{w(2^{j+1}B)}\bigg(\int_{2^{j+1}B}\big|f(y)\big|w(y)\,dy\bigg)\\
&\leq C\big\|f\big\|_{\mathcal M^{1,\psi}(w)}\cdot\frac{\psi(w(2^{j+1}B))}{w(2^{j+1}B)}.
\end{split}
\end{equation*}
Consequently,
\begin{equation*}
\begin{split}
I'_2&\leq C\big\|f\big\|_{\mathcal M^{1,\psi}(w)}
\times\sum_{j=1}^\infty\frac{\psi(w(2^{j+1}B))}{\psi(w(B))}\cdot\frac{w(B)}{w(2^{j+1}B)}.
\end{split}
\end{equation*}
Recall that $w\in A_1\subset A_\infty$, therefore, by using the $\mathcal D_\kappa$ condition (\ref{D condition}) of $\psi$ again, the inequality (\ref{compare}) and the fact that $0\leq\kappa<1$, we get
\begin{align}\label{psi2}
\sum_{j=1}^\infty\frac{\psi(w(2^{j+1}B))}{\psi(w(B))}\cdot\frac{w(B)}{w(2^{j+1}B)}
&\leq C\sum_{j=1}^\infty\frac{w(B)^{1-\kappa}}{w(2^{j+1}B)^{1-\kappa}}\notag\\
&\leq C\sum_{j=1}^\infty\left(\frac{|B|}{|2^{j+1}B|}\right)^{\delta^*(1-\kappa)}\notag\\
&\leq C\sum_{j=1}^\infty\left(\frac{1}{2^{(j+1)n}}\right)^{\delta^*(1-\kappa)}\notag\\
&\leq C,
\end{align}
which implies our desired estimate $I'_2\leq C\big\|f\big\|_{\mathcal M^{1,\psi}(w)}$. Summing up the above estimates for $I'_1$ and $I'_2$, and then taking the supremum over all balls $B\subset\mathbb R^n$ and all $\sigma>0$, we finish the proof of Theorem \ref{mainthm:2}.
\end{proof}

\section{Proof of Theorems \ref{mainthm:3} and \ref{mainthm:4}}

To prove our main theorems in this section, we need the following lemma about $BMO$ functions.
\begin{lemma}\label{BMO}
Let $b$ be a function in $BMO(\mathbb R^n)$. Then

$(i)$ For every ball $B$ in $\mathbb R^n$ and for all $j\in\mathbb Z^+$,
\begin{equation*}
\big|b_{2^{j+1}B}-b_B\big|\leq C\cdot(j+1)\|b\|_*.
\end{equation*}

$(ii)$ For every ball $B$ in $\mathbb R^n$ and for all $w\in A_p$ with $1\leq p<\infty$,
\begin{equation*}
\bigg(\int_B\big|b(x)-b_B\big|^pw(x)\,dx\bigg)^{1/p}\leq C\|b\|_*\cdot w(B)^{1/p}.
\end{equation*}
\end{lemma}
\begin{proof}
For the proof of $(i)$, we refer the reader to \cite{stein2}. For the proof of $(ii)$, we refer the reader to \cite{wang}.
\end{proof}

\begin{proof}[Proof of Theorem $\ref{mainthm:3}$]
Let $f\in\mathcal M^{p,\psi}(w)$ with $1<p<\infty$ and $w\in A_p$. For each fixed ball $B=B(x_0,r_B)\subset\mathbb R^n$, as before, we represent $f$ as $f=f_1+f_2$, where $f_1=f\cdot\chi_{2B}$, $2B=B(x_0,2r_B)\subset\mathbb R^n$. By the linearity of the commutator operator $[b,T_{\theta}]$, we write
\begin{equation*}
\begin{split}
&\frac{1}{\psi(w(B))^{1/p}}\bigg(\int_B\big|[b,T_\theta](f)(x)\big|^pw(x)\,dx\bigg)^{1/p}\\
\leq\ &\frac{1}{\psi(w(B))^{1/p}}\bigg(\int_B\big|[b,T_\theta](f_1)(x)\big|^pw(x)\,dx\bigg)^{1/p}\\
&+\frac{1}{\psi(w(B))^{1/p}}\bigg(\int_B\big|[b,T_\theta](f_2)(x)\big|^pw(x)\,dx\bigg)^{1/p}\\
:=\ &J_1+J_2.
\end{split}
\end{equation*}
Since $T_{\theta}$ is bounded on $L^p_w(\mathbb R^n)$ for $1<p<\infty$ and $w\in A_p$, then by the well-known boundedness criterion for the commutators of linear operators, which was obtained by Alvarez et al.in \cite{alvarez}, we know that $[b,T_{\theta}]$ is also bounded on $L^p_w(\mathbb R^n)$ for all $1<p<\infty$ and $w\in A_p$, whenever $b\in BMO(\mathbb R^n)$. This fact together with the $\mathcal D_\kappa$ condition (\ref{D condition}) of $\psi$ and inequality (\ref{weights}) imply
\begin{equation*}
\begin{split}
J_1&\leq \frac{1}{\psi(w(B))^{1/p}}
\big\|[b,T_\theta](f_1)\big\|_{L^p_w}\\
&\leq C\cdot\frac{1}{\psi(w(B))^{1/p}}
\bigg(\int_{2B}\big|f(x)\big|^p w(x)\,dx\bigg)^{1/p}\\
&\leq C\big\|f\big\|_{\mathcal M^{p,\psi}(w)}
\cdot\frac{\psi(w(2B))^{1/p}}{\psi(w(B))^{1/p}}\\
&\leq C\big\|f\big\|_{\mathcal M^{p,\psi}(w)}\cdot\frac{w(2B)^{\kappa/p}}{w(B)^{\kappa/p}}\\
&\leq C\big\|f\big\|_{\mathcal M^{p,\psi}(w)}.
\end{split}
\end{equation*}
Let us now turn to the estimate of $J_2$. By definition, for any $x\in B$, we have
\begin{equation*}
\big|[b,T_\theta](f_2)(x)\big|\leq\big|b(x)-b_{B}\big|\cdot\big|T_\theta(f_2)(x)\big|
+\Big|T_\theta\big([b_{B}-b]f_2\big)(x)\Big|.
\end{equation*}
In the proof of Theorem \ref{mainthm:1}, we have already shown that (see \eqref{pointwise1})
\begin{equation*}
\big|T_\theta(f_2)(x)\big|\leq C\sum_{j=1}^\infty\frac{1}{|2^{j+1}B|}\int_{2^{j+1}B}\big|f(y)\big|\,dy.
\end{equation*}
Following the same arguments as in \eqref{pointwise1}, we can also prove that
\begin{equation}\label{pointwise2}
\begin{split}
\Big|T_\theta\big([b_{B}-b]f_2\big)(x)\Big|
&\leq\int_{\mathbb R^n}\frac{|[b_B-b(y)]f_2(y)|}{|x-y|^n}dy\\
&\leq C\int_{(2B)^c}\frac{|[b_B-b(y)]f(y)|}{|x_0-y|^n}dy\\
&\leq C\sum_{j=1}^\infty\frac{1}{|2^{j+1}B|}\int_{2^{j+1}B}\big|b(y)-b_{B}\big|\cdot\big|f(y)\big|\,dy.
\end{split}
\end{equation}
Hence, from the above two pointwise estimates for $\big|T_\theta(f_2)(x)\big|$ and $\big|T_\theta\big([b_{B}-b]f_2\big)(x)\big|$, it follows that
\begin{equation*}
\begin{split}
J_2&\leq\frac{C}{\psi(w(B))^{1/p}}\bigg(\int_B\big|b(x)-b_B\big|^pw(x)\,dx\bigg)^{1/p}
\times\bigg(\sum_{j=1}^\infty\frac{1}{|2^{j+1}B|}\int_{2^{j+1}B}\big|f(y)\big|\,dy\bigg)\\
&+C\cdot\frac{w(B)^{1/p}}{\psi(w(B))^{1/p}}\sum_{j=1}^\infty
\frac{1}{|2^{j+1}B|}\int_{2^{j+1}B}\big|b_{2^{j+1}B}-b_B\big|\cdot\big|f(y)\big|\,dy\\
&+C\cdot\frac{w(B)^{1/p}}{\psi(w(B))^{1/p}}\sum_{j=1}^\infty
\frac{1}{|2^{j+1}B|}\int_{2^{j+1}B}\big|b(y)-b_{2^{j+1}B}\big|\cdot\big|f(y)\big|\,dy\\
&:=J_3+J_4+J_5.
\end{split}
\end{equation*}
Below we will give the estimates of $J_3$, $J_4$ and $J_5$, respectively. Using $(ii)$ of Lemma \ref{BMO}, H\"older's inequality and the $A_p$ condition, we obtain
\begin{equation*}
\begin{split}
J_3&\leq C\|b\|_*\cdot\frac{w(B)^{1/p}}{\psi(w(B))^{1/p}}
\times\bigg(\sum_{j=1}^\infty\frac{1}{|2^{j+1}B|}\int_{2^{j+1}B}\big|f(y)\big|\,dy\bigg)\\
&\leq C\|b\|_*\cdot\frac{w(B)^{1/p}}{\psi(w(B))^{1/p}}
\sum_{j=1}^\infty\frac{1}{|2^{j+1}B|}
\bigg(\int_{2^{j+1}B}\big|f(y)\big|^pw(y)\,dy\bigg)^{1/p}\\
&\times\left(\int_{2^{j+1}B}w(y)^{-{p'}/p}\,dy\right)^{1/{p'}}\\
&\leq C\big\|f\big\|_{\mathcal M^{p,\psi}(w)}
\times\sum_{j=1}^\infty\frac{\psi(w(2^{j+1}B))^{1/p}}{\psi(w(B))^{1/p}}\cdot\frac{w(B)^{1/p}}{w(2^{j+1}B)^{1/p}}\\
&\leq C\big\|f\big\|_{\mathcal M^{p,\psi}(w)},
\end{split}
\end{equation*}
where in the last inequality we have used the estimate \eqref{psi1}. Applying $(i)$ of Lemma \ref{BMO}, H\"older's inequality and the $A_p$ condition, we can deduce that
\begin{equation*}
\begin{split}
J_4&\leq C\|b\|_*\cdot\frac{w(B)^{1/p}}{\psi(w(B))^{1/p}}
\times\sum_{j=1}^\infty\frac{(j+1)}{|2^{j+1}B|}\int_{2^{j+1}B}\big|f(y)\big|\,dy\\
&\leq C\|b\|_*\cdot\frac{w(B)^{1/p}}{\psi(w(B))^{1/p}}
\sum_{j=1}^\infty\frac{(j+1)}{|2^{j+1}B|}
\bigg(\int_{2^{j+1}B}\big|f(y)\big|^pw(y)\,dy\bigg)^{1/p}\\
&\times\left(\int_{2^{j+1}B}w(y)^{-{p'}/p}\,dy\right)^{1/{p'}}\\
&\leq C\big\|f\big\|_{\mathcal M^{p,\psi}(w)}
\times\sum_{j=1}^\infty\big(j+1\big)
\cdot\frac{\psi(w(2^{j+1}B))^{1/p}}{\psi(w(B))^{1/p}}\cdot\frac{w(B)^{1/p}}{w(2^{j+1}B)^{1/p}}.
\end{split}
\end{equation*}
For any $j\in\mathbb Z^+$, since $0<w(B)<w(2^{j+1}B)<+\infty$ when $w\in A_p$ with $1<p<\infty$, then by using the $\mathcal D_\kappa$ condition (\ref{D condition}) of $\psi$ and the inequality (\ref{compare}) together with the fact that $0\leq\kappa<1$, we thus obtain
\begin{align}\label{psi3}
\sum_{j=1}^\infty\big(j+1\big)\cdot\frac{\psi(w(2^{j+1}B))^{1/p}}{\psi(w(B))^{1/p}}
\cdot\frac{w(B)^{1/p}}{w(2^{j+1}B)^{1/p}}
&\leq C\sum_{j=1}^\infty\big(j+1\big)\cdot\frac{w(B)^{{(1-\kappa)}/p}}{w(2^{j+1}B)^{{(1-\kappa)}/p}}\notag\\
&\leq C\sum_{j=1}^\infty\big(j+1\big)\cdot\left(\frac{|B|}{|2^{j+1}B|}\right)^{\delta {(1-\kappa)}/p}\notag\\
&\leq C\sum_{j=1}^\infty\big(j+1\big)\cdot\left(\frac{1}{2^{(j+1)n}}\right)^{\delta {(1-\kappa)}/p}\notag\\
&\leq C,
\end{align}
where the last series is convergent since the exponent $\delta {(1-\kappa)}/p$ is positive. This implies our desired estimate $J_4\leq C\big\|f\big\|_{\mathcal M^{p,\psi}(w)}$. It remains to estimate the last term $J_5$. An application of H\"older's inequality gives us that
\begin{equation*}
\begin{split}
J_5&\leq C\cdot\frac{w(B)^{1/p}}{\psi(w(B))^{1/p}}\sum_{j=1}^\infty\frac{1}{|2^{j+1}B|}
\bigg(\int_{2^{j+1}B}\big|f(y)\big|^pw(y)\,dy\bigg)^{1/p}\\
&\times\left(\int_{2^{j+1}B}\big|b(y)-b_{2^{j+1}B}\big|^{p'}w(y)^{-{p'}/p}\,dy\right)^{1/{p'}}.
\end{split}
\end{equation*}
If we set $\mu(y)=w(y)^{-{p'}/p}$, then we have $\mu\in A_{p'}$ because $w\in A_p$(see \cite{duoand,garcia}). Thus, it follows from $(ii)$ of Lemma \ref{BMO} and the $A_p$ condition that
\begin{equation*}
\begin{split}
\left(\int_{2^{j+1}B}\big|b(y)-b_{2^{j+1}B}\big|^{p'}\mu(y)\,dy\right)^{1/{p'}}
&\leq C\|b\|_*\cdot\mu\big(2^{j+1}B\big)^{1/{p'}}\\
&=C\|b\|_*\cdot\left(\int_{2^{j+1}B}w(y)^{-{p'}/p}dy\right)^{1/{p'}}\\
&\leq C\|b\|_*\cdot\frac{|2^{j+1}B|}{w(2^{j+1}B)^{1/p}}.
\end{split}
\end{equation*}
Therefore, in view of the estimate \eqref{psi1}, we conclude that
\begin{equation*}
\begin{split}
J_5&\leq C\|b\|_*\cdot\frac{w(B)^{1/p}}{\psi(w(B))^{1/p}}
\sum_{j=1}^\infty\frac{1}{w(2^{j+1}B)^{1/p}}\bigg(\int_{2^{j+1}B}\big|f(y)\big|^pw(y)\,dy\bigg)^{1/p}\\
&\leq C\big\|f\big\|_{\mathcal M^{p,\psi}(w)}\times
\sum_{j=1}^\infty\frac{\psi(w(2^{j+1}B))^{1/p}}{\psi(w(B))^{1/p}}\cdot\frac{w(B)^{1/p}}{w(2^{j+1}B)^{1/p}}\\
&\leq C\big\|f\big\|_{\mathcal M^{p,\psi}(w)}.
\end{split}
\end{equation*}
Summarizing the estimates derived above and then taking the supremum over all balls $B\subset\mathbb R^n$, we complete the proof of Theorem \ref{mainthm:3}.
\end{proof}

\begin{proof}[Proof of Theorem $\ref{mainthm:4}$]
For any fixed ball $B=B(x_0,r_B)$ in $\mathbb R^n$, as before, we represent $f$ as $f=f_1+f_2$, where $f_1=f\cdot\chi_{2B}$, $2B=B(x_0,2r_B)\subset\mathbb R^n$. Then for any given $\sigma>0$, by the linearity of the commutator operator $[b,T_{\theta}]$, one can write
\begin{equation*}
\begin{split}
&\frac{1}{\psi(w(B))}\cdot w\big(\big\{x\in B:\big|[b,T_\theta](f)(x)\big|>\sigma\big\}\big)\\
\leq &\frac{1}{\psi(w(B))}\cdot w\big(\big\{x\in B:\big|[b,T_\theta](f_1)(x)\big|>\sigma/2\big\}\big)\\
&+\frac{1}{\psi(w(B))}\cdot w\big(\big\{x\in B:\big|[b,T_\theta](f_2)(x)\big|>\sigma/2\big\}\big)\\
:=&J'_1+J'_2.
\end{split}
\end{equation*}
By using Theorem \ref{commutator} and the previous estimate \eqref{main esti2}, we get
\begin{equation*}
\begin{split}
J'_1&\leq C\cdot\frac{1}{\psi(w(B))}\int_{\mathbb R^n}\Phi\left(\frac{|f_1(x)|}{\sigma}\right)\cdot w(x)\,dx\\
&= C\cdot\frac{1}{\psi(w(B))}\int_{2B}\Phi\left(\frac{|f(x)|}{\sigma}\right)\cdot w(x)\,dx\\
&= C\cdot\frac{\psi(w(2B))}{\psi(w(B))}\cdot\frac{1}{\psi(w(2B))}
\int_{2B}\Phi\left(\frac{|f(x)|}{\sigma}\right)\cdot w(x)\,dx\\
&\leq C\cdot\frac{\psi(w(2B))}{\psi(w(B))}\cdot\frac{w(2B)}{\psi(w(2B))}
\cdot\bigg\|\Phi\left(\frac{|f|}{\,\sigma\,}\right)\bigg\|_{L\log L(w),2B}.
\end{split}
\end{equation*}
Moreover, since $0<w(B)<w(2B)<+\infty$ when $w\in A_1$, then by the $\mathcal D_\kappa$ condition (\ref{D condition}) of $\psi$ and inequality (\ref{weights}), we have
\begin{equation*}
\begin{split}
J'_1&\leq C\cdot\frac{w(2B)^\kappa}{w(B)^\kappa}\cdot
\left\{\frac{w(2B)}{\psi(w(2B))}\cdot\bigg\|\Phi\left(\frac{|f|}{\,\sigma\,}\right)\bigg\|_{L\log L(w),2B}\right\}\\
&\leq C\cdot\bigg\|\Phi\left(\frac{|f|}{\,\sigma\,}\right)\bigg\|_{\mathcal M^{1,\psi}_{L\log L}(w)},
\end{split}
\end{equation*}
which is our desired estimate. We now turn to deal with the term $J'_2$. Recall that the following inequality
\begin{equation*}
\big|[b,T_\theta](f_2)(x)\big|\leq\big|b(x)-b_{B}\big|\cdot\big|T_\theta(f_2)(x)\big|
+\Big|T_\theta\big([b_{B}-b]f_2\big)(x)\Big|
\end{equation*}
is valid. So we can further decompose $J'_2$ as
\begin{equation*}
\begin{split}
J'_2\leq&\frac{1}{\psi(w(B))}\cdot
w\big(\big\{x\in B:\big|b(x)-b_{B}\big|\cdot\big|T_\theta(f_2)(x)\big|>\sigma/4\big\}\big)\\
&+\frac{1}{\psi(w(B))}\cdot
w\Big(\Big\{x\in B:\Big|T_\theta\big([b_{B}-b]f_2\big)(x)\Big|>\sigma/4\Big\}\Big)\\
:=&J'_3+J'_4.
\end{split}
\end{equation*}
By using the previous pointwise estimate \eqref{pointwise1}, Chebyshev's inequality together with $(ii)$ of Lemma \ref{BMO}, we deduce that
\begin{equation*}
\begin{split}
J'_3&\leq\frac{1}{\psi(w(B))}\cdot\frac{\,4\,}{\sigma}
\int_B\big|b(x)-b_{B}\big|\cdot\big|T_\theta(f_2)(x)\big|w(x)\,dx\\
&\leq C\sum_{j=1}^\infty\frac{1}{|2^{j+1}B|}\int_{2^{j+1}B}\frac{|f(y)|}{\sigma}\,dy
\times\frac{1}{\psi(w(B))}\cdot\int_B\big|b(x)-b_{B}\big|w(x)\,dx\\
&\leq C\|b\|_*\sum_{j=1}^\infty\frac{1}{|2^{j+1}B|}\int_{2^{j+1}B}\frac{|f(y)|}{\sigma}\,dy
\times\frac{w(B)}{\psi(w(B))}.\\
\end{split}
\end{equation*}
Furthermore, note that $t\leq\Phi(t)=t\cdot(1+\log^+t)$ for any $t>0$. It then follows from the $A_1$ condition and the previous estimate \eqref{main esti1} that
\begin{equation*}
\begin{split}
J'_3&\leq C\sum_{j=1}^\infty\frac{1}{w(2^{j+1}B)}\int_{2^{j+1}B}\frac{|f(y)|}{\sigma}\cdot w(y)\,dy
\times\frac{w(B)}{\psi(w(B))}\\
&\leq C\sum_{j=1}^\infty\frac{1}{w(2^{j+1}B)}\int_{2^{j+1}B}\Phi\left(\frac{|f(y)|}{\sigma}\right)\cdot w(y)\,dy
\times\frac{w(B)}{\psi(w(B))}\\
&\leq C\sum_{j=1}^\infty\bigg\|\Phi\left(\frac{|f|}{\,\sigma\,}\right)\bigg\|_{L\log L(w),2^{j+1}B}
\times\frac{w(B)}{\psi(w(B))}\\
&= C\sum_{j=1}^\infty\left\{\frac{w(2^{j+1}B)}{\psi(w(2^{j+1}B))}\cdot
\bigg\|\Phi\left(\frac{|f|}{\,\sigma\,}\right)\bigg\|_{L\log L(w),2^{j+1}B}\right\}
\times\frac{\psi(w(2^{j+1}B))}{\psi(w(B))}\cdot\frac{w(B)}{w(2^{j+1}B)}\\
&\leq C\cdot\bigg\|\Phi\left(\frac{|f|}{\,\sigma\,}\right)\bigg\|_{\mathcal M^{1,\psi}_{L\log L}(w)}
\times\sum_{j=1}^\infty\frac{\psi(w(2^{j+1}B))}{\psi(w(B))}\cdot\frac{w(B)}{w(2^{j+1}B)}\\
&\leq C\cdot\bigg\|\Phi\left(\frac{|f|}{\,\sigma\,}\right)\bigg\|_{\mathcal M^{1,\psi}_{L\log L}(w)},
\end{split}
\end{equation*}
where in the last inequality we have used the estimate \eqref{psi2}. On the other hand,
applying the pointwise estimate \eqref{pointwise2} and Chebyshev's inequality, we have
\begin{equation*}
\begin{split}
J'_4&\leq\frac{1}{\psi(w(B))}\cdot\frac{\,4\,}{\sigma}
\int_B\Big|T_{\theta}\big([b_{B}-b]f_2\big)(x)\Big|w(x)\,dx\\
&\leq\frac{w(B)}{\psi(w(B))}\cdot\frac{\,C\,}{\sigma}
\sum_{j=1}^\infty\frac{1}{|2^{j+1}B|}\int_{2^{j+1}B}\big|b(y)-b_{B}\big|\cdot\big|f(y)\big|\,dy\\
&\leq\frac{w(B)}{\psi(w(B))}\cdot\frac{\,C\,}{\sigma}
\sum_{j=1}^\infty\frac{1}{|2^{j+1}B|}\int_{2^{j+1}B}\big|b(y)-b_{2^{j+1}B}\big|\cdot\big|f(y)\big|\,dy\\
&+\frac{w(B)}{\psi(w(B))}\cdot\frac{\,C\,}{\sigma}
\sum_{j=1}^\infty\frac{1}{|2^{j+1}B|}\int_{2^{j+1}B}\big|b_{2^{j+1}B}-b_B\big|\cdot\big|f(y)\big|\,dy\\
&:=J'_5+J'_6.
\end{split}
\end{equation*}
For the term $J'_5$, since $w\in A_1$, it follows from the $A_1$ condition and the fact $t\leq \Phi(t)$ that
\begin{equation*}
\begin{split}
J'_5&\leq\frac{\,C\,}{\sigma}\cdot\frac{w(B)}{\psi(w(B))}
\sum_{j=1}^\infty\frac{1}{w(2^{j+1}B)}\int_{2^{j+1}B}\big|b(y)-b_{2^{j+1}B}\big|\cdot\big|f(y)\big|w(y)\,dy\\
&\leq C\cdot\frac{w(B)}{\psi(w(B))}
\sum_{j=1}^\infty\frac{1}{w(2^{j+1}B)}\int_{2^{j+1}B}\big|b(y)-b_{2^{j+1}B}\big|
\cdot\Phi\left(\frac{|f(y)|}{\sigma}\right)w(y)\,dy.\\
\end{split}
\end{equation*}
Furthermore, we use the generalized H\"older's inequality with weight \eqref{Wholder} to obtain
\begin{equation*}
\begin{split}
J'_5&\leq C\cdot\frac{w(B)}{\psi(w(B))}
\sum_{j=1}^\infty\big\|b-b_{2^{j+1}B}\big\|_{\exp L(w),2^{j+1}B}
\bigg\|\Phi\left(\frac{|f|}{\,\sigma\,}\right)\bigg\|_{L\log L(w),2^{j+1}B}\\
&\leq C\|b\|_*\cdot\frac{w(B)}{\psi(w(B))}
\sum_{j=1}^\infty\bigg\|\Phi\left(\frac{|f|}{\,\sigma\,}\right)\bigg\|_{L\log L(w),2^{j+1}B}.
\end{split}
\end{equation*}
In the last inequality, we have used the well-known fact that (see \cite{zhang})
\begin{equation}\label{Jensen}
\big\|b-b_{B}\big\|_{\exp L(w),B}\leq C\|b\|_*,\qquad \mbox{for any ball }B\subset\mathbb R^n.
\end{equation}
It is equivalent to the inequality
\begin{equation*}
\frac{1}{w(B)}\int_B\exp\bigg(\frac{|b(y)-b_B|}{c_0\|b\|_*}\bigg)w(y)\,dy\leq C,
\end{equation*}
which is just a corollary of the well-known John--Nirenberg's inequality (see \cite{john}) and the comparison property of $A_1$ weights. Hence, by the estimate \eqref{psi2}
\begin{equation*}
\begin{split}
J'_5&\leq C\|b\|_*\sum_{j=1}^\infty\left\{\frac{w(2^{j+1}B)}{\psi(w(2^{j+1}B))}\cdot
\bigg\|\Phi\left(\frac{|f|}{\,\sigma\,}\right)\bigg\|_{L\log L(w),2^{j+1}B}\right\}\\
&\times\frac{\psi(w(2^{j+1}B))}{\psi(w(B))}\cdot\frac{w(B)}{w(2^{j+1}B)}\\
&\leq C\cdot\bigg\|\Phi\left(\frac{|f|}{\,\sigma\,}\right)\bigg\|_{\mathcal M^{1,\psi}_{L\log L}(w)}
\times\sum_{j=1}^\infty\frac{\psi(w(2^{j+1}B))}{\psi(w(B))}\cdot\frac{w(B)}{w(2^{j+1}B)}\\
&\leq C\cdot\bigg\|\Phi\left(\frac{|f|}{\,\sigma\,}\right)\bigg\|_{\mathcal M^{1,\psi}_{L\log L}(w)}.
\end{split}
\end{equation*}
For the last term $J'_6$ we proceed as follows. Using $(i)$ of Lemma \ref{BMO} together with the facts $w\in A_1$ and $t\leq\Phi(t)=t\cdot(1+\log^+t)$, we deduce that
\begin{equation*}
\begin{split}
J'_6&\leq C\cdot\frac{w(B)}{\psi(w(B))}
\sum_{j=1}^\infty(j+1)\|b\|_*\cdot\frac{1}{|2^{j+1}B|}\int_{2^{j+1}B}\frac{|f(y)|}{\sigma}\,dy\\
&\leq C\cdot\frac{w(B)}{\psi(w(B))}
\sum_{j=1}^\infty(j+1)\|b\|_*\cdot\frac{1}{w(2^{j+1}B)}\int_{2^{j+1}B}\frac{|f(y)|}{\sigma}\cdot w(y)\,dy\\
&\leq C\|b\|_*\cdot\frac{w(B)}{\psi(w(B))}
\sum_{j=1}^\infty\frac{(j+1)}{w(2^{j+1}B)}\int_{2^{j+1}B}\Phi\left(\frac{|f(y)|}{\sigma}\right)\cdot w(y)\,dy\\
&= C\|b\|_*\sum_{j=1}^\infty\left\{\frac{w(2^{j+1}B)}{\psi(w(2^{j+1}B))}\cdot
\bigg\|\Phi\left(\frac{|f|}{\,\sigma\,}\right)\bigg\|_{L\log L(w),2^{j+1}B}\right\}\\
&\times(j+1)\cdot\frac{\psi(w(2^{j+1}B))}{\psi(w(B))}\cdot\frac{w(B)}{w(2^{j+1}B)}\\
&\leq C\cdot\bigg\|\Phi\left(\frac{|f|}{\,\sigma\,}\right)\bigg\|_{\mathcal M^{1,\psi}_{L\log L}(w)}
\times\sum_{j=1}^\infty(j+1)\cdot\frac{\psi(w(2^{j+1}B))}{\psi(w(B))}\cdot\frac{w(B)}{w(2^{j+1}B)}.
\end{split}
\end{equation*}
Recall that $w\in A_1\subset A_\infty$. We can now argue exactly as we did in the estimation of \eqref{psi3} to get
\begin{align}\label{psi4}
\sum_{j=1}^\infty(j+1)\cdot\frac{\psi(w(2^{j+1}B))}{\psi(w(B))}\cdot\frac{w(B)}{w(2^{j+1}B)}
&\leq C\sum_{j=1}^\infty(j+1)\cdot\frac{w(B)^{1-\kappa}}{w(2^{j+1}B)^{1-\kappa}}\notag\\
&\leq C\sum_{j=1}^\infty(j+1)\cdot\left(\frac{|B|}{|2^{j+1}B|}\right)^{\delta^*(1-\kappa)}\notag\\
&\leq C\sum_{j=1}^\infty(j+1)\cdot\left(\frac{1}{2^{(j+1)n}}\right)^{\delta^*(1-\kappa)}\notag\\
&\leq C,
\end{align}
Notice that the exponent $\delta^*{(1-\kappa)}$ is positive by the choice of $\kappa$, which guarantees that the last series is convergent. If we substitute this estimate \eqref{psi4} into the term $J'_6$, we get the desired inequality
\begin{equation*}
J'_6\leq C\cdot\bigg\|\Phi\left(\frac{|f|}{\,\sigma\,}\right)\bigg\|_{\mathcal M^{1,\psi}_{L\log L}(w)}.
\end{equation*}
This completes the proof of Theorem \ref{mainthm:4}.
\end{proof}

\section{Partial results on two-weight problems}
In the last section, we consider related problems about two-weight, weak type $(p,p)$ inequalities with $1<p<\infty$. Let $\mathcal T$ be the classical Calder¨®n--Zygmund operator with standard kernel, that is, $\mathcal T=T_{\theta}$ when $\theta(t)=t^{\delta}$ with $0<\delta\leq1$. It is well known that $\mathcal T$ is a bounded operator on $L^p_w(\mathbb R^n)$ for all $1<p<\infty$ and $w\in A_p$, and of course, $\mathcal T$ is a bounded operator from $L^p_w(\mathbb R^n)$ into $WL^p_w(\mathbb R^n)$. In the two-weight context, however, the $A_p$ condition is NOT sufficient for the weak-type $(p,p)$ inequality for $\mathcal T$. More precisely, given a pair of weights $(u,v)$ and $p$, $1<p<\infty$, the weak-type inequality
\begin{equation}\label{T}
u\big(\big\{x\in\mathbb R^n:\big|\mathcal Tf(x)\big|>\sigma\big\}\big)
\leq \frac{C}{\sigma^p}\int_{\mathbb R^n}\big|f(x)\big|^p v(x)\,dx
\end{equation}
does not hold if $(u,v)\in A_p$: there exists a positive constant $C$ such that for every cube $Q\subset\mathbb R^n$,
\begin{equation}\label{two}
\left(\frac1{|Q|}\int_Q u(x)\,dx\right)^{1/p}\left(\frac1{|Q|}\int_Q v(x)^{-p'/p}\,dx\right)^{1/{p'}}\leq C<\infty,
\end{equation}
one can see \cite{cruz1,muckenhoupt} for some counter-examples. Here all cubes are assumed to have their sides parallel to the coordinate axes, $Q(x_0,\ell)$ will denote the cube centered at $x_0$ and has side length $\ell$. In \cite{cruz1,cruz2}, Cruz-Uribe and P\'erez considered the problem of finding sufficient conditions on a pair of weights $(u,v)$ such that $\mathcal T$ satisfies the weak-type $(p,p)$ inequality \eqref{T} ($1<p<\infty$). They showed in \cite{cruz2} that if we strengthened the $A_p$ condition \eqref{two} by adding a ``power bump" to the left-hand term, then inequality \eqref{T} holds for all $f\in L^p_v(\mathbb R^n)$. More specifically, if there exists a number $r>1$ such that for every cube $Q$ in $\mathbb R^n$,
\begin{equation}\label{assump1.1}
\left(\frac{1}{|Q|}\int_Q u(x)^r\,dx\right)^{1/{(rp)}}\left(\frac{1}{|Q|}\int_Q v(x)^{-p'/p}\,dx\right)^{1/{p'}}\leq C<\infty,
\end{equation}
then the classical Calder¨®n--Zygmund operator $\mathcal T$ is bounded from $L^p_v(\mathbb R^n)$ into $WL^p_u(\mathbb R^n)$. Moreover, in \cite{cruz1}, the authors improved this result by replacing the ``power bump" in \eqref{assump1.1} by a smaller ``Orlicz bump". To be more precise, they introduced the following $A_p$-type condition in the scale of Orlicz spaces:
\begin{equation*}
\big\|u\big\|_{L(\log L)^{p-1+\delta},Q}^{1/{p}}\left(\frac{1}{|Q|}\int_Q v(x)^{-p'/p}\,dx\right)^{1/{p'}}\leq C<\infty,\qquad \delta>0,
\end{equation*}
where $\big\|u\big\|_{L(\log L)^{p-1+\delta},Q}$ is the mean Luxemburg norm of $u$ on cube $Q$ with Young function $\mathcal A(t)=t\cdot(1+\log^+t)^{p-1+\delta}$. It was shown that inequality \eqref{T} still holds under the $A_p$-type condition on $(u,v)$, and this result is sharp since it does not hold in general when $\delta=0$.

On the other hand, the following Sharp function estimate for $T_{\theta}$ was established in \cite{liu}: there exists some $\delta$, $0<\delta<1$, and a positive constant $C=C_{\delta}$ such that for any $f\in C^\infty_0(\mathbb R^n)$ and $x\in\mathbb R^n$,
\begin{equation}\label{MJ}
\big[M^{\sharp}(|T_{\theta}f|^{\delta})(x)\big]^{1/{\delta}}\leq C M f(x),
\end{equation}
where $M$ is the standard Hardy--Littlewood maximal operator and $M^{\sharp}$ is the well-known Sharp maximal operator defined as
\begin{equation*}
M^{\sharp}f(x):=\sup_{x\in Q}\frac{1}{|Q|}\int_Q\big|f(y)-f_Q\big|\,dy.
\end{equation*}
Here the supremum is taken over all the cubes containing $x$ and $f_Q$ denotes the mean value of $f$ over $Q$, namely, $f_Q=\frac{1}{|Q|}\int_Q f(x)\,dx$. It was pointed out in \cite{cruz2} (Remark 1.3) that by using this Sharp function estimate \eqref{MJ}, we can also show inequality \eqref{T} is true for more general operator $T_{\theta}$, under the condition \eqref{assump1.1} on $(u,v)$. Then we obtain a sufficient condition for $T_{\theta}$ to be weak $(p,p)$ with $1<p<\infty$.
\begin{theorem}\label{WT}
Let $1<p<\infty$. Given a pair of weights $(u,v)$, suppose that for some $r>1$ and for all cubes $Q$,
\begin{equation*}
\left(\frac{1}{|Q|}\int_Q u(x)^r\,dx\right)^{1/{(rp)}}\left(\frac{1}{|Q|}\int_Q v(x)^{-p'/p}\,dx\right)^{1/{p'}}\leq C<\infty.
\end{equation*}
Then the $\theta$-type Calder\'on--Zygmund operator $T_{\theta}$ satisfies the weak-type $(p,p)$ inequality
\begin{equation*}
u\big(\big\{x\in\mathbb R^n:\big|T_{\theta}f(x)\big|>\sigma\big\}\big)
\leq \frac{C}{\sigma^p}\int_{\mathbb R^n}\big|f(x)\big|^p v(x)\,dx,
\end{equation*}
where $C$ does not depend on $f$ and $\sigma>0$.
\end{theorem}

We want to extend Theorem \ref{WT} to the Morrey type spaces. In order to do so, we need to define Morrey type spaces associated to $\psi$ with two weights.
\begin{defn}
Let $1\leq p<\infty$, $0\leq\kappa<1$ and $\psi$ satisfy the $\mathcal D_\kappa$ condition $(\ref{D condition})$. For two weights $u$ and $v$, we denote by $\mathcal M^{p,\psi}(v,u)$ the generalized weighted Morrey space, the space of all locally integrable functions $f$ with finite norm.
\begin{equation*}
\mathcal M^{p,\psi}(v,u):=\Big\{f\in L^p_{loc}(v):\big\|f\big\|_{\mathcal M^{p,\psi}(v,u)}<\infty\Big\},
\end{equation*}
where the norm is given by
\begin{equation*}
\big\|f\big\|_{\mathcal M^{p,\psi}(v,u)}
:=\sup_Q\left(\frac{1}{\psi(u(Q))}\int_Q |f(x)|^pv(x)\,dx\right)^{1/p}.
\end{equation*}
\end{defn}
Note that
\begin{itemize}
  \item If $u=v=w$, then $\mathcal M^{p,\psi}(v,u)$ is the space $\mathcal M^{p,\psi}(w)$ in Definition \ref{WMorrey};
  \item If $\psi(x)=x^{\kappa}$ with $0<\kappa<1$, then $\mathcal M^{p,\psi}(v,u)$ is just the weighted Morrey space with two weights $\mathcal L^{p,\kappa}(v,u)$, which was introduced by Komori and Shirai in \cite{komori}.
\end{itemize}

We are now ready to prove the following result.

\begin{theorem}\label{mainthm:5}
Let $1<p<\infty$ and $u\in A_\infty$. Given a pair of weights $(u,v)$, suppose that for some $r>1$ and for all cubes $Q$,
\begin{equation*}
\left(\frac{1}{|Q|}\int_Q u(x)^r\,dx\right)^{1/{(rp)}}\left(\frac{1}{|Q|}\int_Q v(x)^{-p'/p}\,dx\right)^{1/{p'}}\leq C<\infty.
\end{equation*}
If $\psi$ satisfies the $\mathcal D_\kappa$ condition $(\ref{D condition})$ with $0\leq\kappa<1$, then the $\theta$-type Calder\'on--Zygmund operator $T_{\theta}$ is bounded from $\mathcal M^{p,\psi}(v,u)$ into $W\mathcal M^{p,\psi}(u)$.
\end{theorem}

\begin{proof}[Proof of Theorem $\ref{mainthm:5}$]
Let $f\in\mathcal M^{p,\psi}(v,u)$ with $1<p<\infty$. For any cube $Q=Q(x_0,\ell)\subset\mathbb R^n$ and $\lambda>0$, we will denote by $\lambda Q$ the cube concentric with $Q$ whose each edge is $\lambda$ times as long, that is, $\lambda Q=Q(x_0,\lambda\ell)$. Let
\begin{equation*}
f=f\cdot\chi_{2Q}+f\cdot\chi_{(2Q)^c}:=f_1+f_2,
\end{equation*}
where $\chi_{2Q}$ denotes the characteristic function of $2Q=Q(x_0,2\ell)$. Then for any given $\sigma>0$, we write
\begin{equation*}
\begin{split}
&\frac{1}{\psi(u(Q))^{1/p}}\sigma\cdot
\Big[u\big(\big\{x\in Q:\big|T_{\theta}(f)(x)\big|>\sigma\big\}\big)\Big]^{1/p}\\
\leq &\frac{1}{\psi(u(Q))^{1/p}}\sigma\cdot
\Big[u\big(\big\{x\in Q:\big|T_{\theta}(f_1)(x)\big|>\sigma/2\big\}\big)\Big]^{1/p}\\
&+\frac{1}{\psi(u(Q))^{1/p}}\sigma\cdot
\Big[u\big(\big\{x\in Q:\big|T_{\theta}(f_2)(x)\big|>\sigma/2\big\}\big)\Big]^{1/p}\\
:=&K_1+K_2.
\end{split}
\end{equation*}
Using Theorem \ref{WT}, the $\mathcal D_\kappa$ condition (\ref{D condition}) of $\psi$ and inequality (\ref{weights})(consider cube $Q$ instead of ball $B$), we get
\begin{equation*}
\begin{split}
K_1&\leq C\cdot\frac{1}{\psi(u(Q))^{1/p}}\left(\int_{\mathbb R^n}\big|f_1(x)\big|^p v(x)\,dx\right)^{1/p}\\
&=C\cdot\frac{1}{\psi(u(Q))^{1/p}}\left(\int_{2Q}\big|f(x)\big|^p v(x)\,dx\right)^{1/p}\\
&\leq C\big\|f\big\|_{\mathcal M^{p,\psi}(v,u)}\cdot\frac{\psi(u(2Q))^{1/p}}{\psi(u(Q))^{1/p}}\\
&\leq C\big\|f\big\|_{\mathcal M^{p,\psi}(v,u)}\cdot\frac{u(2Q)^{\kappa/p}}{u(Q)^{\kappa/p}}\\
&\leq C\big\|f\big\|_{\mathcal M^{p,\psi}(v,u)}.
\end{split}
\end{equation*}
As for the term $K_2$, using the same methods and steps as we deal with $I_2$ in Theorem \ref{mainthm:1}, we can also obtain that for any $x\in Q$,
\begin{equation}\label{Ttheta}
\big|T_{\theta}(f_2)(x)\big|\leq C\sum_{j=1}^\infty\frac{1}{|2^{j+1}Q|}\int_{2^{j+1}Q}\big|f(y)\big|\,dy.
\end{equation}
This pointwise estimate together with Chebyshev's inequality implies
\begin{equation*}
\begin{split}
K_2&\leq\frac{2}{\psi(u(Q))^{1/p}}\cdot\left(\int_Q\big|T_{\theta}(f_2)(x)\big|^pu(x)\,dx\right)^{1/p}\\
&\leq C\cdot\frac{u(Q)^{1/p}}{\psi(u(Q))^{1/p}}
\sum_{j=1}^\infty\frac{1}{|2^{j+1}Q|}\int_{2^{j+1}Q}\big|f(y)\big|\,dy.
\end{split}
\end{equation*}
Moreover, an application of H\"older's inequality gives that
\begin{equation*}
\begin{split}
K_2&\leq C\cdot\frac{u(Q)^{1/p}}{\psi(u(Q))^{1/p}}
\sum_{j=1}^\infty\frac{1}{|2^{j+1}Q|}\left(\int_{2^{j+1}Q}\big|f(y)\big|^pv(y)\,dy\right)^{1/p}\\
&\times\left(\int_{2^{j+1}Q}v(y)^{-p'/p}\,dy\right)^{1/{p'}}\\
&\leq C\big\|f\big\|_{\mathcal M^{p,\psi}(v,u)}\cdot\frac{u(Q)^{1/p}}{\psi(u(Q))^{1/p}} \sum_{j=1}^\infty\frac{\psi(u(2^{j+1}Q))^{1/p}}{|2^{j+1}Q|}
\times\left(\int_{2^{j+1}Q}v(y)^{-p'/p}\,dy\right)^{1/{p'}}.
\end{split}
\end{equation*}
For any $j\in\mathbb Z^+$, since $0<u(Q)<u(2^{j+1}Q)<+\infty$ when $u$ is a weight function, then by the $\mathcal D_\kappa$ condition (\ref{D condition}) of $\psi$ with $0\leq\kappa<1$, we can see that
\begin{equation}\label{comparepsi}
\frac{\psi(u(2^{j+1}Q))^{1/p}}{\psi(u(Q))^{1/p}}\leq\frac{u(2^{j+1}Q)^{\kappa/p}}{u(Q)^{\kappa/p}}.
\end{equation}
In addition, we apply H\"older's inequality with exponent $r$ to get
\begin{equation}\label{U}
u\big(2^{j+1}Q\big)=\int_{2^{j+1}Q}u(y)\,dy
\leq\big|2^{j+1}Q\big|^{1/{r'}}\left(\int_{2^{j+1}Q}u(y)^r\,dy\right)^{1/r}.
\end{equation}
Hence, in view of \eqref{comparepsi} and \eqref{U} derived above, we have
\begin{equation*}
\begin{split}
K_2&\leq C\big\|f\big\|_{\mathcal M^{p,\psi}(v,u)}\sum_{j=1}^\infty
\frac{u(Q)^{{(1-\kappa)}/p}}{u(2^{j+1}Q)^{{(1-\kappa)}/p}}\cdot\frac{u(2^{j+1}Q)^{1/p}}{|2^{j+1}Q|}
\times\left(\int_{2^{j+1}Q}v(y)^{-p'/p}\,dy\right)^{1/{p'}}\\
&\leq C\big\|f\big\|_{\mathcal M^{p,\psi}(v,u)}\sum_{j=1}^\infty
\frac{u(Q)^{{(1-\kappa)}/p}}{u(2^{j+1}Q)^{{(1-\kappa)}/p}}\cdot\frac{|2^{j+1}Q|^{1/{(r'p)}}}{|2^{j+1}Q|}\\
&\times\left(\int_{2^{j+1}Q}u(y)^r\,dy\right)^{1/{(rp)}}\left(\int_{2^{j+1}Q}v(y)^{-p'/p}\,dy\right)^{1/{p'}}\\
&\leq C\big\|f\big\|_{\mathcal M^{p,\psi}(v,u)}\times\sum_{j=1}^\infty\frac{u(Q)^{{(1-\kappa)}/p}}{u(2^{j+1}Q)^{{(1-\kappa)}/p}}.
\end{split}
\end{equation*}
The last inequality is obtained by the condition \eqref{assump1.1} on $(u,v)$. Furthermore, by our additional hypothesis on $u:u\in A_\infty$ and inequality (\ref{compare})(consider cube $Q$ instead of ball $B$), we get
\begin{align}\label{5}
\sum_{j=1}^\infty\frac{u(Q)^{{(1-\kappa)}/p}}{u(2^{j+1}Q)^{{(1-\kappa)}/p}}
&\leq C\sum_{j=1}^\infty\left(\frac{|Q|}{|2^{j+1}Q|}\right)^{{\delta(1-\kappa)}/p}\notag\\
&\leq C\sum_{j=1}^\infty\left(\frac{1}{2^{(j+1)n}}\right)^{{\delta(1-\kappa)}/p}\notag\\
&\leq C,
\end{align}
which implies our desired estimate $K_2\leq C\big\|f\big\|_{\mathcal M^{p,\psi}(v,u)}$. Summing up the above estimates for $K_1$ and $K_2$, and then taking the supremum over all cubes $Q\subset\mathbb R^n$ and all $\sigma>0$, we finish the proof of Theorem \ref{mainthm:5}.
\end{proof}

Let $M$ denote the Hardy--Littlewood maximal operator and $M^{\sharp}$ denote the Sharp maximal operator. For $\delta>0$, we define
\begin{equation*}
M_{\delta}(f):=\big[M(|f|^{\delta})\big]^{1/{\delta}},\qquad  M^{\sharp}_{\delta}(f):=\big[M^{\sharp}(|f|^{\delta})\big]^{1/{\delta}}.
\end{equation*}
The maximal function associated to $\mathcal A(t)=t(1+\log^+t)$ is defined as
\begin{equation*}
M_{L\log L}f(x):=\sup_{x\in Q}\big\|f\big\|_{L\log L,Q},
\end{equation*}
where the supremum is taken over all the cubes containing $x$. Let $b\in BMO(\mathbb R^n)$ and $[b,T_{\theta}]$ be the commutator of the $\theta$-type Calder\'on--Zygmund operator. In \cite{liu}, it was proved that if $\theta$ satisfies condition $(\ref{theta2})$, then for $0<\delta<\varepsilon<1$, there exists a positive constant $C=C_{\delta,\varepsilon}$ such that for any $f\in C^\infty_0(\mathbb R^n)$ and $x\in\mathbb R^n$,
\begin{equation}\label{MJ2}
M^{\sharp}_\delta([b,T_{\theta}]f)(x)\leq C\|b\|_*\Big(M_{\varepsilon}(T_{\theta}f)(x)+M_{L\log L}f(x)\Big).
\end{equation}
Using this Sharp function estimate \eqref{MJ2} and following the idea of the proof in \cite{cruz2}, we can also establish the two-weight, weak-type norm inequality for $[b,T_{\theta}]$.
\begin{theorem}\label{WT2}
Let $1<p<\infty$ and $b\in BMO(\mathbb R^n)$. Given a pair of weights $(u,v)$, suppose that for some $r>1$ and for all cubes $Q$,
\begin{equation*}
\left(\frac{1}{|Q|}\int_Q u(x)^r\,dx\right)^{1/{(rp)}}\big\|v^{-1/p}\big\|_{\mathcal A,Q}\leq C<\infty,
\end{equation*}
where $\mathcal A(t)=t^{p'}(1+\log^+t)^{p'}$ is a Young function. If $\theta$ satisfies $(\ref{theta2})$, then the commutator operator $[b,T_{\theta}]$ satisfies the weak-type $(p,p)$ inequality
\begin{equation*}
u\big(\big\{x\in\mathbb R^n:\big|[b,T_{\theta}]f(x)\big|>\sigma\big\}\big)
\leq \frac{C}{\sigma^p}\int_{\mathbb R^n}\big|f(x)\big|^p v(x)\,dx,
\end{equation*}
where $C>0$ does not depend on $f$ and $\sigma>0$.
\end{theorem}

We will extend Theorem \ref{WT2} to the Morrey type spaces. In order to do so, we need the following key lemma.
\begin{lemma}\label{three}
Given three Young functions $\mathcal A$, $\mathcal B$ and $\mathcal C$ such that for all $t>0$,
\begin{equation*}
\mathcal A^{-1}(t)\cdot\mathcal B^{-1}(t)\leq\mathcal C^{-1}(t),
\end{equation*}
where $\mathcal A^{-1}(t)$ is the inverse function of $\mathcal A(t)$. Then we have the following generalized H\"older's inequality due to O'Neil \cite{neil}: for any cube $Q\subset\mathbb R^n$ and all functions $f$ and $g$,
\begin{equation*}
\big\|f\cdot g\big\|_{\mathcal C,Q}\leq 2\big\|f\big\|_{\mathcal A,Q}\big\|g\big\|_{\mathcal B,Q}.
\end{equation*}
\end{lemma}

\begin{theorem}\label{mainthm:6}
Let $1<p<\infty$, $u\in A_\infty$ and $b\in BMO(\mathbb R^n)$. Given a pair of weights $(u,v)$, suppose that for some $r>1$ and for all cubes $Q$,
\begin{equation}\label{assump1.2}
\left(\frac{1}{|Q|}\int_Q u(x)^r\,dx\right)^{1/{(rp)}}\big\|v^{-1/p}\big\|_{\mathcal A,Q}\leq C<\infty,
\end{equation}
where $\mathcal A(t)=t^{p'}(1+\log^+t)^{p'}$. If $\theta$ satisfies $(\ref{theta2})$ and $\psi$ satisfies the $\mathcal D_\kappa$ condition $(\ref{D condition})$ with $0\leq\kappa<1$, then the commutator operator $[b,T_{\theta}]$ is bounded from $\mathcal M^{p,\psi}(v,u)$ into $W\mathcal M^{p,\psi}(u)$.
\end{theorem}

\begin{proof}[Proof of Theorem $\ref{mainthm:6}$]
Let $f\in\mathcal M^{p,\psi}(v,u)$ with $1<p<\infty$. For an arbitrary cube $Q=Q(x_0,\ell)$ in $\mathbb R^n$, as before, we set
\begin{equation*}
f=f_1+f_2,\qquad f_1=f\cdot\chi_{2Q},\quad  f_2=f\cdot\chi_{(2Q)^c}.
\end{equation*}
Then for any given $\sigma>0$, we write
\begin{equation*}
\begin{split}
&\frac{1}{\psi(u(Q))^{1/p}}\sigma\cdot
\Big[u\big(\big\{x\in Q:\big|[b,T_{\theta}](f)(x)\big|>\sigma\big\}\big)\Big]^{1/p}\\
\leq &\frac{1}{\psi(u(Q))^{1/p}}\sigma\cdot
\Big[u\big(\big\{x\in Q:\big|[b,T_{\theta}](f_1)(x)\big|>\sigma/2\big\}\big)\Big]^{1/p}\\
&+\frac{1}{\psi(u(Q))^{1/p}}\sigma\cdot
\Big[u\big(\big\{x\in Q:\big|[b,T_{\theta}](f_2)(x)\big|>\sigma/2\big\}\big)\Big]^{1/p}\\
:=&K'_1+K'_2.
\end{split}
\end{equation*}
Using Theorem \ref{WT2}, the $\mathcal D_\kappa$ condition (\ref{D condition}) of $\psi$ and inequality (\ref{weights})(consider cube $Q$ instead of ball $B$), we get
\begin{equation*}
\begin{split}
K'_1&\leq C\cdot\frac{1}{\psi(u(Q))^{1/p}}\left(\int_{\mathbb R^n}\big|f_1(x)\big|^p v(x)\,dx\right)^{1/p}\\
&=C\cdot\frac{1}{\psi(u(Q))^{1/p}}\left(\int_{2Q}\big|f(x)\big|^p v(x)\,dx\right)^{1/p}\\
&\leq C\big\|f\big\|_{\mathcal M^{p,\psi}(v,u)}\cdot\frac{\psi(u(2Q))^{1/p}}{\psi(u(Q))^{1/p}}\\
&\leq C\big\|f\big\|_{\mathcal M^{p,\psi}(v,u)}\cdot\frac{u(2Q)^{\kappa/p}}{u(Q)^{\kappa/p}}\\
&\leq C\big\|f\big\|_{\mathcal M^{p,\psi}(v,u)}.
\end{split}
\end{equation*}
Next we estimate $K'_2$. For any $x\in Q$, from the definition of $[b,T_{\theta}]$, we can see that
\begin{equation*}
\begin{split}
\big|[b,T_{\theta}](f_2)(x)\big|
&\leq \big|b(x)-b_{Q}\big|\cdot\big|T_\theta(f_2)(x)\big|
+\Big|T_\theta\big([b_{Q}-b]f_2\big)(x)\Big|\\
&:=\xi(x)+\eta(x).
\end{split}
\end{equation*}
Thus we have
\begin{equation*}
\begin{split}
K'_2\leq&\frac{1}{\psi(u(Q))^{1/p}}\sigma\cdot\Big[u\big(\big\{x\in Q:\xi(x)>\sigma/4\big\}\big)\Big]^{1/p}\\
&+\frac{1}{\psi(u(Q))^{1/p}}\sigma\cdot\Big[u\big(\big\{x\in Q:\eta(x)>\sigma/4\big\}\big)\Big]^{1/p}\\
:=&K'_3+K'_4.
\end{split}
\end{equation*}
For the term $K'_3$, it follows from the pointwise estimate \eqref{Ttheta} mentioned above and Chebyshev's inequality that
\begin{equation*}
\begin{split}
K'_3&\leq\frac{4}{\psi(u(Q))^{1/p}}\cdot\left(\int_Q\big|\xi(x)\big|^pu(x)\,dx\right)^{1/p}\\
&\leq\frac{C}{\psi(u(Q))^{1/p}}\cdot\left(\int_Q\big|b(x)-b_{Q}\big|^pu(x)\,dx\right)^{1/p}
\times\bigg(\sum_{j=1}^\infty\frac{1}{|2^{j+1}Q|}\int_{2^{j+1}Q}\big|f(y)\big|\,dy\bigg)\\
&\leq C\cdot\frac{u(Q)^{1/p}}{\psi(u(Q))^{1/p}}
\sum_{j=1}^\infty\frac{1}{|2^{j+1}Q|}\int_{2^{j+1}Q}\big|f(y)\big|\,dy,
\end{split}
\end{equation*}
where in the last inequality we have used the fact that Lemma \ref{BMO}$(ii)$ still holds when $u$ is an $A_{\infty}$ weight with $B$ replaced by $Q$. Repeating the arguments in the proof of Theorem \ref{mainthm:5}, we can show that $K'_3\leq C\big\|f\big\|_{\mathcal M^{p,\psi}(v,u)}$. As for the term $K'_4$,
using the same methods and steps as we deal with $J_2$ in Theorem \ref{mainthm:3}, we can show the following pointwise estimate as well.
\begin{equation*}
\eta(x)=\Big|T_\theta\big([b_{Q}-b]f_2\big)(x)\Big|\leq C\sum_{j=1}^\infty\frac{1}{|2^{j+1}Q|}\int_{2^{j+1}Q}\big|b(y)-b_{Q}\big|\cdot\big|f(y)\big|\,dy.
\end{equation*}
This, together with Chebyshev's inequality yields
\begin{equation*}
\begin{split}
K'_4&\leq\frac{4}{\psi(u(Q))^{1/p}}\cdot\left(\int_Q\big|\eta(x)\big|^pu(x)\,dx\right)^{1/p}\\
&\leq C\cdot\frac{u(Q)^{1/p}}{\psi(u(Q))^{1/p}}\cdot
\sum_{j=1}^\infty\frac{1}{|2^{j+1}Q|}\int_{2^{j+1}Q}\big|b(y)-b_{Q}\big|\cdot\big|f(y)\big|\,dy\\
&\leq C\cdot\frac{u(Q)^{1/p}}{\psi(u(Q))^{1/p}}\cdot
\sum_{j=1}^\infty\frac{1}{|2^{j+1}Q|}\int_{2^{j+1}Q}\big|b(y)-b_{{2^{j+1}Q}}\big|\cdot\big|f(y)\big|\,dy\\
&+C\cdot\frac{u(Q)^{1/p}}{\psi(u(Q))^{1/p}}\cdot
\sum_{j=1}^\infty\frac{1}{|2^{j+1}Q|}\int_{2^{j+1}Q}\big|b_{{2^{j+1}Q}}-b_{Q}\big|\cdot\big|f(y)\big|\,dy\\
&:=K'_5+K'_6.
\end{split}
\end{equation*}
An application of H\"older's inequality leads to that
\begin{equation*}
\begin{split}
K'_5&\leq C\cdot\frac{u(Q)^{1/p}}{\psi(u(Q))^{1/p}}\cdot
\sum_{j=1}^\infty\frac{1}{|2^{j+1}Q|}\left(\int_{2^{j+1}Q}\big|f(y)\big|^pv(y)\,dy\right)^{1/p}\\
&\times\left(\int_{2^{j+1}Q}\big|b(y)-b_{{2^{j+1}Q}}\big|^{p'}v(y)^{-p'/p}\,dy\right)^{1/{p'}}\\
&\leq C\big\|f\big\|_{\mathcal M^{p,\psi}(v,u)}\cdot\frac{u(Q)^{1/p}}{\psi(u(Q))^{1/p}}\cdot \sum_{j=1}^\infty\frac{\psi(u(2^{j+1}Q))^{1/p}}{|2^{j+1}Q|}\\
&\times\big|2^{j+1}Q\big|^{1/{p'}}\Big\|(b-b_{{2^{j+1}Q}})\cdot v^{-1/p}\Big\|_{\mathcal C,2^{j+1}Q},
\end{split}
\end{equation*}
where $\mathcal C(t)=t^{p'}$ is a Young function. For $1<p<\infty$, we know the inverse function of $\mathcal C(t)$ is $\mathcal C^{-1}(t)=t^{1/{p'}}$. Observe that
\begin{equation*}
\begin{split}
\mathcal C^{-1}(t)&=t^{1/{p'}}\\
&=\frac{t^{1/{p'}}}{1+\log^+ t}\times\big(1+\log^+t\big)\\
&=\mathcal A^{-1}(t)\cdot\mathcal B^{-1}(t),
\end{split}
\end{equation*}
where
\begin{equation*}
\mathcal A(t)\approx t^{p'}(1+\log^+t)^{p'},\qquad \mbox{and}\qquad \mathcal B(t)\approx \exp(t).
\end{equation*}
Thus, by Lemma \ref{three} and the estimate \eqref{Jensen}(when $w\equiv1$), we have
\begin{equation*}
\begin{split}
\Big\|(b-b_{{2^{j+1}Q}})\cdot v^{-1/p}\Big\|_{\mathcal C,2^{j+1}Q}
&\leq C\Big\|b-b_{{2^{j+1}Q}}\Big\|_{\mathcal B,2^{j+1}Q}\cdot\Big\|v^{-1/p}\Big\|_{\mathcal A,2^{j+1}Q}\\
&\leq C\|b\|_*\cdot\Big\|v^{-1/p}\Big\|_{\mathcal A,2^{j+1}Q}.
\end{split}
\end{equation*}
Moreover, in view of \eqref{comparepsi} and \eqref{U}, we can deduce that
\begin{equation*}
\begin{split}
K'_5&\leq C\|b\|_*\big\|f\big\|_{\mathcal M^{p,\psi}(v,u)}
\sum_{j=1}^\infty\frac{u(2^{j+1}Q)^{\kappa/p}}{u(Q)^{\kappa/p}}\cdot\frac{u(Q)^{1/p}}{|2^{j+1}Q|^{1/p}}
\cdot\Big\|v^{-1/p}\Big\|_{\mathcal A,2^{j+1}Q}\\
&\leq C\|b\|_*\big\|f\big\|_{\mathcal M^{p,\psi}(v,u)}
\sum_{j=1}^\infty\frac{u(Q)^{{(1-\kappa)}/p}}{u(2^{j+1}Q)^{{(1-\kappa)}/p}}\\
&\times\left(\frac{1}{|2^{j+1}Q|}\int_{2^{j+1}Q}u(x)^r\,dx\right)^{1/{(rp)}}
\cdot\Big\|v^{-1/p}\Big\|_{\mathcal A,2^{j+1}Q}\\
&\leq C\big\|f\big\|_{\mathcal M^{p,\psi}(v,u)}
\sum_{j=1}^\infty\frac{u(Q)^{{(1-\kappa)}/p}}{u(2^{j+1}Q)^{{(1-\kappa)}/p}}\\
&\leq C\big\|f\big\|_{\mathcal M^{p,\psi}(v,u)}.
\end{split}
\end{equation*}
The last inequality is obtained by the condition \eqref{assump1.2} on $(u,v)$ and the estimate \eqref{5}.
It remains to estimate the last term $K'_6$. Applying Lemma \ref{BMO}$(i)$(use $Q$ instead of $B$) and H\"older's inequality, we get
\begin{equation*}
\begin{split}
K'_6&\leq C\cdot\frac{u(Q)^{1/p}}{\psi(u(Q))^{1/p}}
\sum_{j=1}^\infty\frac{(j+1)\|b\|_*}{|2^{j+1}Q|}\int_{2^{j+1}Q}\big|f(y)\big|\,dy\\
&\leq C\cdot\frac{u(Q)^{1/p}}{\psi(u(Q))^{1/p}}
\sum_{j=1}^\infty\frac{(j+1)\|b\|_*}{|2^{j+1}Q|}\left(\int_{2^{j+1}Q}\big|f(y)\big|^pv(y)\,dy\right)^{1/p}\\
&\times\left(\int_{2^{j+1}Q}v(y)^{-p'/p}\,dy\right)^{1/{p'}}\\
&\leq C\big\|f\big\|_{\mathcal M^{p,\psi}(v,u)}\cdot\frac{u(Q)^{1/p}}{\psi(u(Q))^{1/p}} \sum_{j=1}^\infty(j+1)\cdot\frac{\psi(u(2^{j+1}Q))^{1/p}}{|2^{j+1}Q|}\\
&\times\left(\int_{2^{j+1}Q}v(y)^{-p'/p}\,dy\right)^{1/{p'}}.
\end{split}
\end{equation*}
Let $\mathcal C(t)$, $\mathcal A(t)$ be the same as before. Obviously, $\mathcal C(t)\leq\mathcal A(t)$ for all $t>0$, then for any cube $Q\subset\mathbb R^n$, we have $\big\|f\big\|_{\mathcal C,Q}\leq\big\|f\big\|_{\mathcal A,Q}$ by definition, which implies that condition \eqref{assump1.2} is stronger that condition \eqref{assump1.1}. This fact together with \eqref{comparepsi} and \eqref{U} yield
\begin{equation*}
\begin{split}
K'_6&\leq C\big\|f\big\|_{\mathcal M^{p,\psi}(v,u)}\sum_{j=1}^\infty(j+1)\cdot
\frac{u(Q)^{{(1-\kappa)}/p}}{u(2^{j+1}Q)^{{(1-\kappa)}/p}}\cdot\frac{u(2^{j+1}Q)^{1/p}}{|2^{j+1}Q|}\\
&\times\left(\int_{2^{j+1}Q}v(y)^{-p'/p}\,dy\right)^{1/{p'}}\\
&\leq C\big\|f\big\|_{\mathcal M^{p,\psi}(v,u)}\sum_{j=1}^\infty(j+1)\cdot
\frac{u(Q)^{{(1-\kappa)}/p}}{u(2^{j+1}Q)^{{(1-\kappa)}/p}}\cdot\frac{|2^{j+1}Q|^{1/{(r'p)}}}{|2^{j+1}Q|}\\
&\times\left(\int_{2^{j+1}Q}u(y)^r\,dy\right)^{1/{(rp)}}\left(\int_{2^{j+1}Q}v(y)^{-p'/p}\,dy\right)^{1/{p'}}\\
&\leq C\big\|f\big\|_{\mathcal M^{p,\psi}(v,u)}
\sum_{j=1}^\infty(j+1)\cdot\frac{u(Q)^{{(1-\kappa)}/p}}{u(2^{j+1}Q)^{{(1-\kappa)}/p}}.
\end{split}
\end{equation*}
Moreover, by our additional hypothesis on $u:u\in A_\infty$ and inequality (\ref{compare})(use $Q$ instead of $B$), we finally obtain
\begin{equation*}
\begin{split}
\sum_{j=1}^\infty(j+1)\cdot\frac{u(Q)^{{(1-\kappa)}/p}}{u(2^{j+1}Q)^{{(1-\kappa)}/p}}
&\leq C\sum_{j=1}^\infty(j+1)\cdot\left(\frac{|Q|}{|2^{j+1}Q|}\right)^{{\delta(1-\kappa)}/p}\\
&\leq C\sum_{j=1}^\infty(j+1)\cdot\left(\frac{1}{2^{(j+1)n}}\right)^{{\delta(1-\kappa)}/p}\\
&\leq C,
\end{split}
\end{equation*}
which in turn gives that $K'_6\leq C\big\|f\big\|_{\mathcal M^{p,\psi}(v,u)}$. Summing up all the above estimates, and then taking the supremum over all cubes $Q\subset\mathbb R^n$ and all $\sigma>0$, we therefore conclude the proof of Theorem \ref{mainthm:6}.
\end{proof}

In particular, if we take $\psi(x)=x^\kappa$ with $0<\kappa<1$, then we immediately
get the following two-weight, weak type $(p,p)$ inequalities for $T_{\theta}$ and $[b,T_{\theta}]$ in the weighted Morrey spaces.
\begin{corollary}
Let $1<p<\infty$, $0<\kappa<1$ and $u\in A_\infty$. Given a pair of weights $(u,v)$, suppose that for some $r>1$ and for all cubes $Q$,
\begin{equation*}
\left(\frac{1}{|Q|}\int_Q u(x)^r\,dx\right)^{1/{(rp)}}\left(\frac{1}{|Q|}\int_Q v(x)^{-p'/p}\,dx\right)^{1/{p'}}
\leq C<\infty.
\end{equation*}
Then the $\theta$-type Calder\'on--Zygmund operator $T_{\theta}$ is bounded from $\mathcal L^{p,\kappa}(v,u)$ into $W\mathcal L^{p,\kappa}(u)$.
\end{corollary}

\begin{corollary}
Let $1<p<\infty$, $0<\kappa<1$, $u\in A_\infty$ and $b\in BMO(\mathbb R^n)$. Given a pair of weights $(u,v)$, suppose that for some $r>1$ and for all cubes $Q$,
\begin{equation*}
\left(\frac{1}{|Q|}\int_Q u(x)^r\,dx\right)^{1/{(rp)}}\big\|v^{-1/p}\big\|_{\mathcal A,Q}\leq C<\infty,
\end{equation*}
where $\mathcal A(t)=t^{p'}(1+\log^+t)^{p'}$. If $\theta$ satisfies $(\ref{theta2})$, then the commutator operator $[b,T_{\theta}]$ is bounded from $\mathcal L^{p,\kappa}(v,u)$ into $W\mathcal L^{p,\kappa}(u)$.
\end{corollary}

\end{document}